\documentclass[final]{siamltex}

\usepackage{amssymb}
\usepackage{amsmath}
\usepackage{stmaryrd}

\newtheorem{satz}{Theorem}

\newtheorem{assumption}[theorem]{Assumption}
\newtheorem{korollar}[theorem]{Corollary}

\DeclareMathOperator{\R}{\mathbb{R}}
\newcommand{\dx}{\, d x}
\newcommand{\JJ}{\widehat J}
\newcommand{\Ug}{U_{\partial \Omega}}
\newcommand{\we}{w^{\varepsilon}}

\title{A note on adjoint error estimation for one-dimensional stationary balance laws with shocks.}
\author{Jochen Sch\"utz, Sebastian Noelle, Christina Steiner and  Georg May}

\begin{document}

\maketitle

\begin{abstract}
We consider one-dimensional steady-state balance laws with discontinuous solutions.
Giles and Pierce \cite{GP00} realized that a shock leads to a new term in the adjoint error representation
for target functionals. This term disappears if and only if the adjoint solution satisfies
an {\em internal boundary condition}.
Curiously, most computer codes implementing adjoint error estimation ignore the new term in the functional, as well as the internal adjoint boundary condition. 
The purpose of this note is to justify this omission as follows: if one represents the exact forward and adjoint solutions as vanishing viscosity limits of the corresponding viscous problems, then the internal boundary condition is naturally satisfied in the limit.
\end{abstract}



\pagestyle{myheadings}
\thispagestyle{plain}
\markboth{SCH\"UTZ ET AL.}{Adjoint Error Estimation with Shocks}

\section{Introduction}
\label{sec:introduction}
We consider stationary one-dimensional conservation laws with source term, also called balance laws \cite{BalanceLaws10},
\begin{align}
 \label{eq:conservation_law}
 f(w)_x + S(w) &= 0 \quad \forall x \in \Omega,
\end{align}
equipped with in- and outflow boundary conditions. Both $f: \R^d \rightarrow \R^d$ and $S: \R^d \rightarrow \R^d$ are given smooth functions, and $\Omega \subset \R$ is a one-dimensional domain. 
One particular example of \eqref{eq:conservation_law} is one-dimensional nozzle flow \cite{AN}. 
In many applications, the user is interested in the value of so-called {\em target functionals}, such as lift and drag coefficients in aerodynamics. In this context, for a given smooth function $p$ of the solution $w$, we consider the functional $\JJ(w) := \int_{\Omega} p(w) \dx$. For smooth exact, respectively approximate, solutions $w$ and $v$, the error
\begin{align}
 \mathcal{E}(v,w) &:= \JJ(v) - \JJ(w)
\intertext{
in the target functional $\JJ$ is given by 
}
 \label{fehler_darstellung}
 \mathcal{E}(v,w) &= \mathcal{R}(z(w),v) + \mathcal{H}(z(w),v,w),
\intertext{
where $\mathcal{H}(z(w),v,w)$ is a higher order term to be discussed in section \ref{sec:adjoint_shock} below, and
}
 \mathcal{R}(z,v) &:= \int_{\Omega} z^T (f(v)_x + S(v)) \ dx
\end{align}
is the inner product of the residual of the approximate solution $v$  and an adjoint solution $z \equiv z(w)$, which is implicitly defined by the system of equations
\begin{alignat}{2}
 \label{eq:adjoint}
 -f'(w)^T z_x + S'(w)^T z &= p'(w) &\quad& \forall x \in \Omega, 
\end{alignat}
subject to suitable boundary conditions. 
The clue of this error representation is that $\mathcal{R}(z,v)$ does not (directly) depend on the unknown solution $w$ and can thus be evaluated numerically, provided one approximates $z$ suitably. Therefore, neglecting $\mathcal{H}$, and localizing the terms in the inner product $\mathcal{R}$, one obtains an a-posteriori error estimate, which can be used to refine the grid in such a way that the target functional is computed accurately at low cost. This strategy has been used by many authors, in particular for steady state computations (see, e.g., \cite{BR97, RHa06, VD2000} and the references therein). 

Giles and Pierce \cite{GP97, GP00} generalized this framework to solutions $w$ and $v$ with 
shocks, and found the  error representation 
\begin{align}
\label{fehler_darstellung_mit_interner_rb}
 \mathcal{E}(v,w) &= \mathcal{R}(z,v) + \bar\alpha\,\mathcal{I}(z,w) + \mathcal{H}(z,v,w) 
\end{align}
where the new term consists of the product of the error in the shock location
$\overline\alpha\equiv\overline\alpha(w,v)$ and the jump term
\begin{align}
 \label{eq:intbc_intro}
 \mathcal{I}(z,w) & :=  - z^T(\alpha) [f(w)_x] - [p(w)].
\end{align}
Here $\alpha\equiv\alpha(w)$ is the shock location, and $[\cdot]$ denotes the jump of a quantity across the shock. 

It is not obvious whether \eqref{eq:intbc_intro} is zero. In the context of time-dependent conservation laws, Giles and Ulbrich \cite{GU08, GU08Teil2} proved convergence of a numerically computed $z_h$ towards $z$ in the framework of a Finite Volume scheme, while in the stationary case, Sch\"utz et al. \cite{SMN10} studied convergence in the context of a Discontinuous Galerkin scheme. The latter result indicates that at least $\mathcal{I}(z_h,w_h)$ converges to zero. In this paper, we will show that under certain conditions, $\mathcal{I}(z,w)$ is indeed zero.

The internal error term presents a serious obstacle to a-posteriori adjoint
error control, since the exact solution, the shock position and the error in the shock location are not known from the data of the computation, and hence the internal error cannot be evaluated. Therefore, in practical computations,  $\bar\alpha\,\mathcal{I}$ is usually neglected. Perhaps surprisingly, this leads to successful adaptive schemes. The aim of this paper is to show analytically that this omission is justified. 

The paper proceeds as follows: In section~\ref{sec:adjoint_shock} we present an alternative, and more detailed, derivation of the rather subtle error representation \eqref{fehler_darstellung_mit_interner_rb} for piecewise smooth solutions.
In section~\ref{sec:convergence_intbc}, we prove that vanishing viscosity solutions $w$ and $z$, provided that $z$ is smooth, satisfy what is called the {\em internal boundary condition}
\begin{align}
 z^T(\alpha) [f(w)_x] = - [p(w)].
\end{align}
Therefore, the internal error term $\bar\alpha\,\mathcal{I}$ vanishes identically, and the a-posteriori error representation is justified for stationary conservation laws with shocks.

\section{Adjoint Error Control in the Discontinuous Case}
\label{sec:adjoint_shock}
In this section, we give an alternative derivation of Giles' and Pierce's  \cite{GP97, GP00} adjoint error representation \eqref{fehler_darstellung_mit_interner_rb} for a non-smooth solution $w$ and a non-smooth function $v$ approximating $w$ in a certain sense we make more precise below. Giles and Pierce use a very short, formal calculus. But the differentiation of nonsmooth solutions, whose discontinuities are in different locations, is rather subtle. Here we confirm their calculation by a more detailed argument: we introduce a one-parameter family of coordinates, which links the smooth regions of both solutions. This helps us to formulate the distance of two such solutions, and to differentiate with respect to the new grid parameter.

We consider the domain $\Omega = [0,1]$. For suitable boundary conditions, it is well-known that solutions to nonlinear conservation laws exhibit jump discontinuities. The location of such a discontinuity (the \emph{shock location}) is denoted by $x = \alpha$.
We assume that $w$ is discontinuous in $x=\alpha$, while it is sufficiently smooth away from $\alpha$, which in particular means that $\lim_{\varepsilon \rightarrow 0^+} w(\alpha \pm \varepsilon) =: w^{\pm}$ exists. This is a standard setting and in no way a restriction. 
Now assume $w$ is perturbed in such a way that the resulting function $v := w + \overline w$ has one (and only one) discontinuity at $x = \beta =: \alpha + \overline \alpha$, and is also smooth away from $\beta$. 
In the case of a non-smooth function $w$, being \emph{approximated} by some other non-smooth function $v$, we cannot simply say that $\|v-w\|_{\infty}$ is small, say $O(\nu)$ for some small parameter $\nu$, because if $\alpha$ and $\beta$ do not coincide, we always have  an $O(1)$ approximation error in the $\infty-$norm in the region between $\alpha$ and $\beta$. 

As a consequence, in the following definition, we state what we mean by a \emph{sufficiently small} perturbation $\overline w$:
\begin{definition}[\emph{Sufficiently close}\label{defsufficientlysmooth} approximation of a discontinuous function] \label{sufficientlyclose} 
We say that $w$ is approximated by $v$ to order $\nu$ if 
\begin{itemize} 
  \item There exist $\alpha, \beta \in \R, 0 < \alpha < 1, 0 < \beta < 1$ and smooth, invertible functions 
  \begin{align}
    \label{def:xi1}
    \xi_1 &: [0,\alpha] \rightarrow [0,\beta], \\
    \label{def:xi2}
    \xi_2 &: [\alpha,1] \rightarrow [\beta,1]
  \end{align} 
  such that we have 
  \begin{alignat}{2}
      \label{def:sufficiently_smooth}
      w(x) &= v(\xi_1(x))^- + O(\nu), &\quad& x < \alpha \\
      w(x) &= v(\xi_2(x))^+ + O(\nu), &\quad& x > \alpha, 
  \end{alignat} 
  with the $O(\nu)-$bound assumed to be uniform.
  \item The $\xi_i$ have to fulfill the properties
	\begin{align} 
	    \label{approx_xistrich1}
	    \frac{d}{dx}\xi_1 &=1 + O(\nu), \\ 
	    \label{approx_xistrich2}
	    \frac{d}{dx}\xi_2 &= 1+O(\nu),  \\
	    \label{approx_xistrich3}
	    \xi_1(\alpha) = \xi_2(\alpha) &= \beta
        \end{align}
 and the second derivatives of $\xi_i$ are bounded. 
  \item The residual $r(v)$ is sufficiently small, meaning that we have the property 
	\begin{align}
	  r(v) := f(v)_x + S(v) &= O(\mu)
	\end{align}
  pointwise except at the discontinuity of $v$, where $\mu$ is another parameter going to zero. Usually, $\mu$ tends much slower to zero than $\nu$ does. 
\end{itemize}
\end{definition}

\begin{assumption}
The results we obtain in this section are independent of the relative position of $\alpha$ and $\beta$. However, for the sake of simplicity, we assume without loss of generality that $\alpha < \beta$.  
\end{assumption}

A visualization of the relevant quantities can be seen in figure \ref{fig:suff_close}. Let us make the remark that functions $w$ and $v$ according to definition \ref{sufficientlyclose} fulfill $\| w(\cdot) - v(\xi(\cdot)) \|_{\infty} = O(\nu)$. 
\begin{figure}[\width=0.45\textwidth]
\setlength{\unitlength}{0.240900pt}
\ifx\plotpoint\undefined\newsavebox{\plotpoint}\fi
\begin{picture}(1650,540)(0,0)
\sbox{\plotpoint}{\rule[-0.200pt]{0.400pt}{0.400pt}}%
\put(131.0,169.0){\rule[-0.200pt]{4.818pt}{0.400pt}}
\put(111,169){\makebox(0,0)[r]{-1}}
\put(1568.0,169.0){\rule[-0.200pt]{4.818pt}{0.400pt}}
\put(131.0,263.0){\rule[-0.200pt]{4.818pt}{0.400pt}}
\put(111,263){\makebox(0,0)[r]{ 0}}
\put(1568.0,263.0){\rule[-0.200pt]{4.818pt}{0.400pt}}
\put(131.0,357.0){\rule[-0.200pt]{4.818pt}{0.400pt}}
\put(111,357){\makebox(0,0)[r]{ 1}}
\put(1568.0,357.0){\rule[-0.200pt]{4.818pt}{0.400pt}}
\put(131.0,452.0){\rule[-0.200pt]{4.818pt}{0.400pt}}
\put(111,452){\makebox(0,0)[r]{ 2}}
\put(1568.0,452.0){\rule[-0.200pt]{4.818pt}{0.400pt}}
\put(131.0,131.0){\rule[-0.200pt]{0.400pt}{4.818pt}}
\put(131,90){\makebox(0,0){0}}
\put(131.0,479.0){\rule[-0.200pt]{0.400pt}{4.818pt}}
\put(495.0,131.0){\rule[-0.200pt]{0.400pt}{4.818pt}}
\put(495,90){\makebox(0,0){0.25}}
\put(495.0,479.0){\rule[-0.200pt]{0.400pt}{4.818pt}}
\put(860.0,131.0){\rule[-0.200pt]{0.400pt}{4.818pt}}
\put(860,90){\makebox(0,0){$\alpha$}}
\put(860.0,479.0){\rule[-0.200pt]{0.400pt}{4.818pt}}
\put(932.0,131.0){\rule[-0.200pt]{0.400pt}{4.818pt}}
\put(932,90){\makebox(0,0){$\beta$}}
\put(932.0,479.0){\rule[-0.200pt]{0.400pt}{4.818pt}}
\put(1224.0,131.0){\rule[-0.200pt]{0.400pt}{4.818pt}}
\put(1224,90){\makebox(0,0){0.75}}
\put(1224.0,479.0){\rule[-0.200pt]{0.400pt}{4.818pt}}
\put(1588.0,131.0){\rule[-0.200pt]{0.400pt}{4.818pt}}
\put(1588.0,479.0){\rule[-0.200pt]{0.400pt}{4.818pt}}
\put(131.0,131.0){\rule[-0.200pt]{0.400pt}{88.651pt}}
\put(131.0,131.0){\rule[-0.200pt]{350.991pt}{0.400pt}}
\put(1588.0,131.0){\rule[-0.200pt]{0.400pt}{88.651pt}}
\put(131.0,499.0){\rule[-0.200pt]{350.991pt}{0.400pt}}
\put(30,315){\makebox(0,0){$y$}}
\put(859,29){\makebox(0,0){$x$}}
\put(1428,459){\makebox(0,0)[r]{$w(x)$}}
\put(1448.0,459.0){\rule[-0.200pt]{12.286pt}{0.400pt}}
\put(131,169){\usebox{\plotpoint}}
\multiput(858.61,169.00)(0.447,41.765){3}{\rule{0.108pt}{25.167pt}}
\multiput(857.17,169.00)(3.000,135.765){2}{\rule{0.400pt}{12.583pt}}
\put(131.0,169.0){\rule[-0.200pt]{175.134pt}{0.400pt}}
\put(861.0,357.0){\rule[-0.200pt]{175.134pt}{0.400pt}}
\sbox{\plotpoint}{\rule[-0.500pt]{1.000pt}{1.000pt}}%
\sbox{\plotpoint}{\rule[-0.200pt]{0.400pt}{0.400pt}}%
\put(1428,418){\makebox(0,0)[r]{$v(x)$}}
\sbox{\plotpoint}{\rule[-0.500pt]{1.000pt}{1.000pt}}%
\multiput(1448,418)(20.756,0.000){3}{\usebox{\plotpoint}}
\put(1499,418){\usebox{\plotpoint}}
\put(131,169){\usebox{\plotpoint}}
\put(131.00,169.00){\usebox{\plotpoint}}
\put(151.76,169.00){\usebox{\plotpoint}}
\put(172.35,170.00){\usebox{\plotpoint}}
\put(193.10,170.00){\usebox{\plotpoint}}
\put(213.70,171.00){\usebox{\plotpoint}}
\put(234.45,171.00){\usebox{\plotpoint}}
\put(255.05,172.00){\usebox{\plotpoint}}
\put(275.80,172.00){\usebox{\plotpoint}}
\put(296.39,173.00){\usebox{\plotpoint}}
\put(317.15,173.00){\usebox{\plotpoint}}
\put(337.91,173.00){\usebox{\plotpoint}}
\put(358.50,174.00){\usebox{\plotpoint}}
\put(379.25,174.00){\usebox{\plotpoint}}
\put(399.86,174.95){\usebox{\plotpoint}}
\put(420.60,175.00){\usebox{\plotpoint}}
\put(441.36,175.00){\usebox{\plotpoint}}
\put(461.95,176.00){\usebox{\plotpoint}}
\put(482.71,176.00){\usebox{\plotpoint}}
\put(503.46,176.00){\usebox{\plotpoint}}
\put(524.10,176.70){\usebox{\plotpoint}}
\put(544.81,177.00){\usebox{\plotpoint}}
\put(565.57,177.00){\usebox{\plotpoint}}
\put(586.32,177.00){\usebox{\plotpoint}}
\put(606.92,178.00){\usebox{\plotpoint}}
\put(627.67,178.00){\usebox{\plotpoint}}
\put(648.43,178.00){\usebox{\plotpoint}}
\put(669.18,178.00){\usebox{\plotpoint}}
\put(689.78,179.00){\usebox{\plotpoint}}
\put(710.53,179.00){\usebox{\plotpoint}}
\put(731.29,179.00){\usebox{\plotpoint}}
\put(752.04,179.00){\usebox{\plotpoint}}
\put(772.80,179.00){\usebox{\plotpoint}}
\put(793.39,180.00){\usebox{\plotpoint}}
\put(814.15,180.00){\usebox{\plotpoint}}
\put(834.90,180.00){\usebox{\plotpoint}}
\put(855.66,180.00){\usebox{\plotpoint}}
\put(876.41,180.00){\usebox{\plotpoint}}
\put(897.17,180.00){\usebox{\plotpoint}}
\put(917.76,181.00){\usebox{\plotpoint}}
\multiput(931,181)(0.360,20.752){8}{\usebox{\plotpoint}}
\put(934.53,354.00){\usebox{\plotpoint}}
\put(955.29,354.00){\usebox{\plotpoint}}
\put(975.88,353.00){\usebox{\plotpoint}}
\put(996.40,352.00){\usebox{\plotpoint}}
\put(1017.10,351.63){\usebox{\plotpoint}}
\put(1037.75,351.00){\usebox{\plotpoint}}
\put(1058.34,350.00){\usebox{\plotpoint}}
\put(1079.10,350.00){\usebox{\plotpoint}}
\put(1099.69,349.00){\usebox{\plotpoint}}
\put(1120.45,349.00){\usebox{\plotpoint}}
\put(1141.20,349.00){\usebox{\plotpoint}}
\put(1161.79,348.00){\usebox{\plotpoint}}
\put(1182.55,348.00){\usebox{\plotpoint}}
\put(1203.31,348.00){\usebox{\plotpoint}}
\put(1224.06,348.00){\usebox{\plotpoint}}
\put(1244.82,348.00){\usebox{\plotpoint}}
\put(1265.57,348.00){\usebox{\plotpoint}}
\put(1286.33,348.00){\usebox{\plotpoint}}
\put(1306.92,349.00){\usebox{\plotpoint}}
\put(1327.68,349.00){\usebox{\plotpoint}}
\put(1348.43,349.00){\usebox{\plotpoint}}
\put(1369.03,350.00){\usebox{\plotpoint}}
\put(1389.78,350.00){\usebox{\plotpoint}}
\put(1410.37,351.00){\usebox{\plotpoint}}
\put(1430.97,352.00){\usebox{\plotpoint}}
\put(1451.72,352.00){\usebox{\plotpoint}}
\put(1472.24,353.00){\usebox{\plotpoint}}
\put(1492.84,354.00){\usebox{\plotpoint}}
\put(1513.51,354.50){\usebox{\plotpoint}}
\put(1534.18,355.00){\usebox{\plotpoint}}
\put(1554.78,356.00){\usebox{\plotpoint}}
\put(1575.37,357.00){\usebox{\plotpoint}}
\put(1588,357){\usebox{\plotpoint}}
\sbox{\plotpoint}{\rule[-0.200pt]{0.400pt}{0.400pt}}%
\put(1428,377){\makebox(0,0)[r]{$\overline w(x) = w(x) - v(x)$}}
\multiput(1448,377)(20.756,0.000){3}{\usebox{\plotpoint}}
\put(1499,377){\usebox{\plotpoint}}
\put(131,263){\usebox{\plotpoint}}
\put(131.00,263.00){\usebox{\plotpoint}}
\put(151.76,263.00){\usebox{\plotpoint}}
\put(172.35,262.00){\usebox{\plotpoint}}
\put(193.10,262.00){\usebox{\plotpoint}}
\put(213.70,261.00){\usebox{\plotpoint}}
\put(234.45,261.00){\usebox{\plotpoint}}
\put(255.05,260.00){\usebox{\plotpoint}}
\put(275.80,260.00){\usebox{\plotpoint}}
\put(296.39,259.00){\usebox{\plotpoint}}
\put(317.15,259.00){\usebox{\plotpoint}}
\put(337.76,258.08){\usebox{\plotpoint}}
\put(358.50,258.00){\usebox{\plotpoint}}
\put(379.25,258.00){\usebox{\plotpoint}}
\put(399.85,257.00){\usebox{\plotpoint}}
\put(420.60,257.00){\usebox{\plotpoint}}
\put(441.36,257.00){\usebox{\plotpoint}}
\put(461.95,256.00){\usebox{\plotpoint}}
\put(482.71,256.00){\usebox{\plotpoint}}
\put(503.46,256.00){\usebox{\plotpoint}}
\put(523.98,255.00){\usebox{\plotpoint}}
\put(544.74,255.00){\usebox{\plotpoint}}
\put(565.49,255.00){\usebox{\plotpoint}}
\put(586.01,254.00){\usebox{\plotpoint}}
\put(606.77,254.00){\usebox{\plotpoint}}
\put(627.52,254.00){\usebox{\plotpoint}}
\put(648.28,254.00){\usebox{\plotpoint}}
\put(668.87,253.00){\usebox{\plotpoint}}
\put(689.63,253.00){\usebox{\plotpoint}}
\put(710.38,253.00){\usebox{\plotpoint}}
\put(731.14,253.00){\usebox{\plotpoint}}
\put(751.89,253.00){\usebox{\plotpoint}}
\put(772.49,252.00){\usebox{\plotpoint}}
\put(793.24,252.00){\usebox{\plotpoint}}
\put(814.00,252.00){\usebox{\plotpoint}}
\put(834.75,252.00){\usebox{\plotpoint}}
\put(855.51,252.00){\usebox{\plotpoint}}
\multiput(858,252)(0.331,20.753){9}{\usebox{\plotpoint}}
\put(878.05,440.00){\usebox{\plotpoint}}
\put(898.80,440.00){\usebox{\plotpoint}}
\put(919.56,440.00){\usebox{\plotpoint}}
\multiput(931,440)(0.358,-20.752){8}{\usebox{\plotpoint}}
\put(935.33,266.00){\usebox{\plotpoint}}
\put(955.92,267.00){\usebox{\plotpoint}}
\put(976.51,268.00){\usebox{\plotpoint}}
\put(997.27,268.00){\usebox{\plotpoint}}
\put(1017.86,269.00){\usebox{\plotpoint}}
\put(1038.46,270.00){\usebox{\plotpoint}}
\put(1059.21,270.00){\usebox{\plotpoint}}
\put(1079.81,271.00){\usebox{\plotpoint}}
\put(1100.56,271.00){\usebox{\plotpoint}}
\put(1121.15,272.00){\usebox{\plotpoint}}
\put(1141.91,272.00){\usebox{\plotpoint}}
\put(1162.67,272.00){\usebox{\plotpoint}}
\put(1183.42,272.00){\usebox{\plotpoint}}
\put(1204.12,272.37){\usebox{\plotpoint}}
\put(1224.77,273.00){\usebox{\plotpoint}}
\put(1245.40,272.20){\usebox{\plotpoint}}
\put(1266.12,272.00){\usebox{\plotpoint}}
\put(1286.87,272.00){\usebox{\plotpoint}}
\put(1307.63,272.00){\usebox{\plotpoint}}
\put(1328.38,272.00){\usebox{\plotpoint}}
\put(1348.98,271.00){\usebox{\plotpoint}}
\put(1369.73,271.00){\usebox{\plotpoint}}
\put(1390.33,270.00){\usebox{\plotpoint}}
\put(1411.08,270.00){\usebox{\plotpoint}}
\put(1431.68,269.00){\usebox{\plotpoint}}
\put(1452.27,268.00){\usebox{\plotpoint}}
\put(1473.02,268.00){\usebox{\plotpoint}}
\put(1493.62,267.00){\usebox{\plotpoint}}
\put(1514.21,266.00){\usebox{\plotpoint}}
\put(1534.80,265.00){\usebox{\plotpoint}}
\put(1555.43,264.19){\usebox{\plotpoint}}
\put(1576.14,263.95){\usebox{\plotpoint}}
\put(1588,263){\usebox{\plotpoint}}
\put(131.0,131.0){\rule[-0.200pt]{0.400pt}{88.651pt}}
\put(131.0,131.0){\rule[-0.200pt]{350.991pt}{0.400pt}}
\put(1588.0,131.0){\rule[-0.200pt]{0.400pt}{88.651pt}}
\put(131.0,499.0){\rule[-0.200pt]{350.991pt}{0.400pt}}
\end{picture}
\begin{picture}(1650,540)(0,0)
\put(131.0,169.0){\rule[-0.200pt]{4.818pt}{0.400pt}}
\put(111,169){\makebox(0,0)[r]{-1}}
\put(1568.0,169.0){\rule[-0.200pt]{4.818pt}{0.400pt}}
\put(131.0,263.0){\rule[-0.200pt]{4.818pt}{0.400pt}}
\put(111,263){\makebox(0,0)[r]{ 0}}
\put(1568.0,263.0){\rule[-0.200pt]{4.818pt}{0.400pt}}
\put(131.0,312.0){\rule[-0.200pt]{4.818pt}{0.400pt}}
\put(111,312){\makebox(0,0)[r]{$\beta$}}
\put(1568.0,312.0){\rule[-0.200pt]{4.818pt}{0.400pt}}
\put(131.0,357.0){\rule[-0.200pt]{4.818pt}{0.400pt}}
\put(111,357){\makebox(0,0)[r]{ 1}}
\put(1568.0,357.0){\rule[-0.200pt]{4.818pt}{0.400pt}}
\put(131.0,452.0){\rule[-0.200pt]{4.818pt}{0.400pt}}
\put(111,452){\makebox(0,0)[r]{ 2}}
\put(1568.0,452.0){\rule[-0.200pt]{4.818pt}{0.400pt}}
\put(131.0,131.0){\rule[-0.200pt]{0.400pt}{4.818pt}}
\put(131,90){\makebox(0,0){0}}
\put(131.0,479.0){\rule[-0.200pt]{0.400pt}{4.818pt}}
\put(495.0,131.0){\rule[-0.200pt]{0.400pt}{4.818pt}}
\put(495,90){\makebox(0,0){0.25}}
\put(495.0,479.0){\rule[-0.200pt]{0.400pt}{4.818pt}}
\put(860.0,131.0){\rule[-0.200pt]{0.400pt}{4.818pt}}
\put(860,90){\makebox(0,0){$\alpha$}}
\put(860.0,479.0){\rule[-0.200pt]{0.400pt}{4.818pt}}
\put(1224.0,131.0){\rule[-0.200pt]{0.400pt}{4.818pt}}
\put(1224,90){\makebox(0,0){0.75}}
\put(1224.0,479.0){\rule[-0.200pt]{0.400pt}{4.818pt}}
\put(1588.0,131.0){\rule[-0.200pt]{0.400pt}{4.818pt}}
\put(1588.0,479.0){\rule[-0.200pt]{0.400pt}{4.818pt}}
\put(131.0,131.0){\rule[-0.200pt]{0.400pt}{88.651pt}}
\put(131.0,131.0){\rule[-0.200pt]{350.991pt}{0.400pt}}
\put(1588.0,131.0){\rule[-0.200pt]{0.400pt}{88.651pt}}
\put(131.0,499.0){\rule[-0.200pt]{350.991pt}{0.400pt}}
\put(30,315){\makebox(0,0){$y$}}
\put(859,29){\makebox(0,0){$x$}}
\put(1428,459){\makebox(0,0)[r]{$\xi(x)$}}
\put(1448.0,459.0){\rule[-0.200pt]{12.286pt}{0.400pt}}
\put(131,263){\usebox{\plotpoint}}
\put(134,262.67){\rule{0.723pt}{0.400pt}}
\multiput(134.00,262.17)(1.500,1.000){2}{\rule{0.361pt}{0.400pt}}
\put(131.0,263.0){\rule[-0.200pt]{0.723pt}{0.400pt}}
\put(149,263.67){\rule{0.482pt}{0.400pt}}
\multiput(149.00,263.17)(1.000,1.000){2}{\rule{0.241pt}{0.400pt}}
\put(137.0,264.0){\rule[-0.200pt]{2.891pt}{0.400pt}}
\put(160,264.67){\rule{0.723pt}{0.400pt}}
\multiput(160.00,264.17)(1.500,1.000){2}{\rule{0.361pt}{0.400pt}}
\put(151.0,265.0){\rule[-0.200pt]{2.168pt}{0.400pt}}
\put(175,265.67){\rule{0.723pt}{0.400pt}}
\multiput(175.00,265.17)(1.500,1.000){2}{\rule{0.361pt}{0.400pt}}
\put(163.0,266.0){\rule[-0.200pt]{2.891pt}{0.400pt}}
\put(186,266.67){\rule{0.723pt}{0.400pt}}
\multiput(186.00,266.17)(1.500,1.000){2}{\rule{0.361pt}{0.400pt}}
\put(178.0,267.0){\rule[-0.200pt]{1.927pt}{0.400pt}}
\put(198,267.67){\rule{0.723pt}{0.400pt}}
\multiput(198.00,267.17)(1.500,1.000){2}{\rule{0.361pt}{0.400pt}}
\put(189.0,268.0){\rule[-0.200pt]{2.168pt}{0.400pt}}
\put(213,268.67){\rule{0.723pt}{0.400pt}}
\multiput(213.00,268.17)(1.500,1.000){2}{\rule{0.361pt}{0.400pt}}
\put(201.0,269.0){\rule[-0.200pt]{2.891pt}{0.400pt}}
\put(224,269.67){\rule{0.723pt}{0.400pt}}
\multiput(224.00,269.17)(1.500,1.000){2}{\rule{0.361pt}{0.400pt}}
\put(216.0,270.0){\rule[-0.200pt]{1.927pt}{0.400pt}}
\put(239,270.67){\rule{0.723pt}{0.400pt}}
\multiput(239.00,270.17)(1.500,1.000){2}{\rule{0.361pt}{0.400pt}}
\put(227.0,271.0){\rule[-0.200pt]{2.891pt}{0.400pt}}
\put(254,271.67){\rule{0.723pt}{0.400pt}}
\multiput(254.00,271.17)(1.500,1.000){2}{\rule{0.361pt}{0.400pt}}
\put(242.0,272.0){\rule[-0.200pt]{2.891pt}{0.400pt}}
\put(265,272.67){\rule{0.723pt}{0.400pt}}
\multiput(265.00,272.17)(1.500,1.000){2}{\rule{0.361pt}{0.400pt}}
\put(257.0,273.0){\rule[-0.200pt]{1.927pt}{0.400pt}}
\put(280,273.67){\rule{0.723pt}{0.400pt}}
\multiput(280.00,273.17)(1.500,1.000){2}{\rule{0.361pt}{0.400pt}}
\put(268.0,274.0){\rule[-0.200pt]{2.891pt}{0.400pt}}
\put(292,274.67){\rule{0.723pt}{0.400pt}}
\multiput(292.00,274.17)(1.500,1.000){2}{\rule{0.361pt}{0.400pt}}
\put(283.0,275.0){\rule[-0.200pt]{2.168pt}{0.400pt}}
\put(306,275.67){\rule{0.723pt}{0.400pt}}
\multiput(306.00,275.17)(1.500,1.000){2}{\rule{0.361pt}{0.400pt}}
\put(295.0,276.0){\rule[-0.200pt]{2.650pt}{0.400pt}}
\put(318,276.67){\rule{0.723pt}{0.400pt}}
\multiput(318.00,276.17)(1.500,1.000){2}{\rule{0.361pt}{0.400pt}}
\put(309.0,277.0){\rule[-0.200pt]{2.168pt}{0.400pt}}
\put(332,277.67){\rule{0.723pt}{0.400pt}}
\multiput(332.00,277.17)(1.500,1.000){2}{\rule{0.361pt}{0.400pt}}
\put(321.0,278.0){\rule[-0.200pt]{2.650pt}{0.400pt}}
\put(347,278.67){\rule{0.723pt}{0.400pt}}
\multiput(347.00,278.17)(1.500,1.000){2}{\rule{0.361pt}{0.400pt}}
\put(335.0,279.0){\rule[-0.200pt]{2.891pt}{0.400pt}}
\put(359,279.67){\rule{0.723pt}{0.400pt}}
\multiput(359.00,279.17)(1.500,1.000){2}{\rule{0.361pt}{0.400pt}}
\put(350.0,280.0){\rule[-0.200pt]{2.168pt}{0.400pt}}
\put(373,280.67){\rule{0.723pt}{0.400pt}}
\multiput(373.00,280.17)(1.500,1.000){2}{\rule{0.361pt}{0.400pt}}
\put(362.0,281.0){\rule[-0.200pt]{2.650pt}{0.400pt}}
\put(388,281.67){\rule{0.723pt}{0.400pt}}
\multiput(388.00,281.17)(1.500,1.000){2}{\rule{0.361pt}{0.400pt}}
\put(376.0,282.0){\rule[-0.200pt]{2.891pt}{0.400pt}}
\put(400,282.67){\rule{0.723pt}{0.400pt}}
\multiput(400.00,282.17)(1.500,1.000){2}{\rule{0.361pt}{0.400pt}}
\put(391.0,283.0){\rule[-0.200pt]{2.168pt}{0.400pt}}
\put(414,283.67){\rule{0.723pt}{0.400pt}}
\multiput(414.00,283.17)(1.500,1.000){2}{\rule{0.361pt}{0.400pt}}
\put(403.0,284.0){\rule[-0.200pt]{2.650pt}{0.400pt}}
\put(429,284.67){\rule{0.723pt}{0.400pt}}
\multiput(429.00,284.17)(1.500,1.000){2}{\rule{0.361pt}{0.400pt}}
\put(417.0,285.0){\rule[-0.200pt]{2.891pt}{0.400pt}}
\put(441,285.67){\rule{0.482pt}{0.400pt}}
\multiput(441.00,285.17)(1.000,1.000){2}{\rule{0.241pt}{0.400pt}}
\put(432.0,286.0){\rule[-0.200pt]{2.168pt}{0.400pt}}
\put(455,286.67){\rule{0.723pt}{0.400pt}}
\multiput(455.00,286.17)(1.500,1.000){2}{\rule{0.361pt}{0.400pt}}
\put(443.0,287.0){\rule[-0.200pt]{2.891pt}{0.400pt}}
\put(470,287.67){\rule{0.723pt}{0.400pt}}
\multiput(470.00,287.17)(1.500,1.000){2}{\rule{0.361pt}{0.400pt}}
\put(458.0,288.0){\rule[-0.200pt]{2.891pt}{0.400pt}}
\put(484,288.67){\rule{0.723pt}{0.400pt}}
\multiput(484.00,288.17)(1.500,1.000){2}{\rule{0.361pt}{0.400pt}}
\put(473.0,289.0){\rule[-0.200pt]{2.650pt}{0.400pt}}
\put(499,289.67){\rule{0.723pt}{0.400pt}}
\multiput(499.00,289.17)(1.500,1.000){2}{\rule{0.361pt}{0.400pt}}
\put(487.0,290.0){\rule[-0.200pt]{2.891pt}{0.400pt}}
\put(511,290.67){\rule{0.482pt}{0.400pt}}
\multiput(511.00,290.17)(1.000,1.000){2}{\rule{0.241pt}{0.400pt}}
\put(502.0,291.0){\rule[-0.200pt]{2.168pt}{0.400pt}}
\put(525,291.67){\rule{0.723pt}{0.400pt}}
\multiput(525.00,291.17)(1.500,1.000){2}{\rule{0.361pt}{0.400pt}}
\put(513.0,292.0){\rule[-0.200pt]{2.891pt}{0.400pt}}
\put(540,292.67){\rule{0.723pt}{0.400pt}}
\multiput(540.00,292.17)(1.500,1.000){2}{\rule{0.361pt}{0.400pt}}
\put(528.0,293.0){\rule[-0.200pt]{2.891pt}{0.400pt}}
\put(554,293.67){\rule{0.723pt}{0.400pt}}
\multiput(554.00,293.17)(1.500,1.000){2}{\rule{0.361pt}{0.400pt}}
\put(543.0,294.0){\rule[-0.200pt]{2.650pt}{0.400pt}}
\put(569,294.67){\rule{0.723pt}{0.400pt}}
\multiput(569.00,294.17)(1.500,1.000){2}{\rule{0.361pt}{0.400pt}}
\put(557.0,295.0){\rule[-0.200pt]{2.891pt}{0.400pt}}
\put(584,295.67){\rule{0.482pt}{0.400pt}}
\multiput(584.00,295.17)(1.000,1.000){2}{\rule{0.241pt}{0.400pt}}
\put(572.0,296.0){\rule[-0.200pt]{2.891pt}{0.400pt}}
\put(598,296.67){\rule{0.723pt}{0.400pt}}
\multiput(598.00,296.17)(1.500,1.000){2}{\rule{0.361pt}{0.400pt}}
\put(586.0,297.0){\rule[-0.200pt]{2.891pt}{0.400pt}}
\put(613,297.67){\rule{0.723pt}{0.400pt}}
\multiput(613.00,297.17)(1.500,1.000){2}{\rule{0.361pt}{0.400pt}}
\put(601.0,298.0){\rule[-0.200pt]{2.891pt}{0.400pt}}
\put(627,298.67){\rule{0.723pt}{0.400pt}}
\multiput(627.00,298.17)(1.500,1.000){2}{\rule{0.361pt}{0.400pt}}
\put(616.0,299.0){\rule[-0.200pt]{2.650pt}{0.400pt}}
\put(642,299.67){\rule{0.723pt}{0.400pt}}
\multiput(642.00,299.17)(1.500,1.000){2}{\rule{0.361pt}{0.400pt}}
\put(630.0,300.0){\rule[-0.200pt]{2.891pt}{0.400pt}}
\put(657,300.67){\rule{0.482pt}{0.400pt}}
\multiput(657.00,300.17)(1.000,1.000){2}{\rule{0.241pt}{0.400pt}}
\put(645.0,301.0){\rule[-0.200pt]{2.891pt}{0.400pt}}
\put(671,301.67){\rule{0.723pt}{0.400pt}}
\multiput(671.00,301.17)(1.500,1.000){2}{\rule{0.361pt}{0.400pt}}
\put(659.0,302.0){\rule[-0.200pt]{2.891pt}{0.400pt}}
\put(686,302.67){\rule{0.723pt}{0.400pt}}
\multiput(686.00,302.17)(1.500,1.000){2}{\rule{0.361pt}{0.400pt}}
\put(674.0,303.0){\rule[-0.200pt]{2.891pt}{0.400pt}}
\put(700,303.67){\rule{0.723pt}{0.400pt}}
\multiput(700.00,303.17)(1.500,1.000){2}{\rule{0.361pt}{0.400pt}}
\put(689.0,304.0){\rule[-0.200pt]{2.650pt}{0.400pt}}
\put(715,304.67){\rule{0.723pt}{0.400pt}}
\multiput(715.00,304.17)(1.500,1.000){2}{\rule{0.361pt}{0.400pt}}
\put(703.0,305.0){\rule[-0.200pt]{2.891pt}{0.400pt}}
\put(730,305.67){\rule{0.482pt}{0.400pt}}
\multiput(730.00,305.17)(1.000,1.000){2}{\rule{0.241pt}{0.400pt}}
\put(718.0,306.0){\rule[-0.200pt]{2.891pt}{0.400pt}}
\put(744,306.67){\rule{0.723pt}{0.400pt}}
\multiput(744.00,306.17)(1.500,1.000){2}{\rule{0.361pt}{0.400pt}}
\put(732.0,307.0){\rule[-0.200pt]{2.891pt}{0.400pt}}
\put(759,307.67){\rule{0.723pt}{0.400pt}}
\multiput(759.00,307.17)(1.500,1.000){2}{\rule{0.361pt}{0.400pt}}
\put(747.0,308.0){\rule[-0.200pt]{2.891pt}{0.400pt}}
\put(773,308.67){\rule{0.723pt}{0.400pt}}
\multiput(773.00,308.17)(1.500,1.000){2}{\rule{0.361pt}{0.400pt}}
\put(762.0,309.0){\rule[-0.200pt]{2.650pt}{0.400pt}}
\put(788,309.67){\rule{0.723pt}{0.400pt}}
\multiput(788.00,309.17)(1.500,1.000){2}{\rule{0.361pt}{0.400pt}}
\put(776.0,310.0){\rule[-0.200pt]{2.891pt}{0.400pt}}
\put(805,310.67){\rule{0.723pt}{0.400pt}}
\multiput(805.00,310.17)(1.500,1.000){2}{\rule{0.361pt}{0.400pt}}
\put(791.0,311.0){\rule[-0.200pt]{3.373pt}{0.400pt}}
\put(820,311.67){\rule{0.723pt}{0.400pt}}
\multiput(820.00,311.17)(1.500,1.000){2}{\rule{0.361pt}{0.400pt}}
\put(808.0,312.0){\rule[-0.200pt]{2.891pt}{0.400pt}}
\put(835,312.67){\rule{0.723pt}{0.400pt}}
\multiput(835.00,312.17)(1.500,1.000){2}{\rule{0.361pt}{0.400pt}}
\put(823.0,313.0){\rule[-0.200pt]{2.891pt}{0.400pt}}
\put(849,313.67){\rule{0.723pt}{0.400pt}}
\multiput(849.00,313.17)(1.500,1.000){2}{\rule{0.361pt}{0.400pt}}
\put(838.0,314.0){\rule[-0.200pt]{2.650pt}{0.400pt}}
\put(867,314.67){\rule{0.723pt}{0.400pt}}
\multiput(867.00,314.17)(1.500,1.000){2}{\rule{0.361pt}{0.400pt}}
\put(852.0,315.0){\rule[-0.200pt]{3.613pt}{0.400pt}}
\put(881,315.67){\rule{0.723pt}{0.400pt}}
\multiput(881.00,315.17)(1.500,1.000){2}{\rule{0.361pt}{0.400pt}}
\put(870.0,316.0){\rule[-0.200pt]{2.650pt}{0.400pt}}
\put(896,316.67){\rule{0.723pt}{0.400pt}}
\multiput(896.00,316.17)(1.500,1.000){2}{\rule{0.361pt}{0.400pt}}
\put(884.0,317.0){\rule[-0.200pt]{2.891pt}{0.400pt}}
\put(914,317.67){\rule{0.482pt}{0.400pt}}
\multiput(914.00,317.17)(1.000,1.000){2}{\rule{0.241pt}{0.400pt}}
\put(899.0,318.0){\rule[-0.200pt]{3.613pt}{0.400pt}}
\put(928,318.67){\rule{0.723pt}{0.400pt}}
\multiput(928.00,318.17)(1.500,1.000){2}{\rule{0.361pt}{0.400pt}}
\put(916.0,319.0){\rule[-0.200pt]{2.891pt}{0.400pt}}
\put(943,319.67){\rule{0.723pt}{0.400pt}}
\multiput(943.00,319.17)(1.500,1.000){2}{\rule{0.361pt}{0.400pt}}
\put(931.0,320.0){\rule[-0.200pt]{2.891pt}{0.400pt}}
\put(960,320.67){\rule{0.723pt}{0.400pt}}
\multiput(960.00,320.17)(1.500,1.000){2}{\rule{0.361pt}{0.400pt}}
\put(946.0,321.0){\rule[-0.200pt]{3.373pt}{0.400pt}}
\put(975,321.67){\rule{0.723pt}{0.400pt}}
\multiput(975.00,321.17)(1.500,1.000){2}{\rule{0.361pt}{0.400pt}}
\put(963.0,322.0){\rule[-0.200pt]{2.891pt}{0.400pt}}
\put(992,322.67){\rule{0.723pt}{0.400pt}}
\multiput(992.00,322.17)(1.500,1.000){2}{\rule{0.361pt}{0.400pt}}
\put(978.0,323.0){\rule[-0.200pt]{3.373pt}{0.400pt}}
\put(1007,323.67){\rule{0.723pt}{0.400pt}}
\multiput(1007.00,323.17)(1.500,1.000){2}{\rule{0.361pt}{0.400pt}}
\put(995.0,324.0){\rule[-0.200pt]{2.891pt}{0.400pt}}
\put(1024,324.67){\rule{0.723pt}{0.400pt}}
\multiput(1024.00,324.17)(1.500,1.000){2}{\rule{0.361pt}{0.400pt}}
\put(1010.0,325.0){\rule[-0.200pt]{3.373pt}{0.400pt}}
\put(1039,325.67){\rule{0.723pt}{0.400pt}}
\multiput(1039.00,325.17)(1.500,1.000){2}{\rule{0.361pt}{0.400pt}}
\put(1027.0,326.0){\rule[-0.200pt]{2.891pt}{0.400pt}}
\put(1057,326.67){\rule{0.723pt}{0.400pt}}
\multiput(1057.00,326.17)(1.500,1.000){2}{\rule{0.361pt}{0.400pt}}
\put(1042.0,327.0){\rule[-0.200pt]{3.613pt}{0.400pt}}
\put(1074,327.67){\rule{0.723pt}{0.400pt}}
\multiput(1074.00,327.17)(1.500,1.000){2}{\rule{0.361pt}{0.400pt}}
\put(1060.0,328.0){\rule[-0.200pt]{3.373pt}{0.400pt}}
\put(1089,328.67){\rule{0.723pt}{0.400pt}}
\multiput(1089.00,328.17)(1.500,1.000){2}{\rule{0.361pt}{0.400pt}}
\put(1077.0,329.0){\rule[-0.200pt]{2.891pt}{0.400pt}}
\put(1106,329.67){\rule{0.723pt}{0.400pt}}
\multiput(1106.00,329.17)(1.500,1.000){2}{\rule{0.361pt}{0.400pt}}
\put(1092.0,330.0){\rule[-0.200pt]{3.373pt}{0.400pt}}
\put(1124,330.67){\rule{0.723pt}{0.400pt}}
\multiput(1124.00,330.17)(1.500,1.000){2}{\rule{0.361pt}{0.400pt}}
\put(1109.0,331.0){\rule[-0.200pt]{3.613pt}{0.400pt}}
\put(1138,331.67){\rule{0.723pt}{0.400pt}}
\multiput(1138.00,331.17)(1.500,1.000){2}{\rule{0.361pt}{0.400pt}}
\put(1127.0,332.0){\rule[-0.200pt]{2.650pt}{0.400pt}}
\put(1156,332.67){\rule{0.723pt}{0.400pt}}
\multiput(1156.00,332.17)(1.500,1.000){2}{\rule{0.361pt}{0.400pt}}
\put(1141.0,333.0){\rule[-0.200pt]{3.613pt}{0.400pt}}
\put(1173,333.67){\rule{0.723pt}{0.400pt}}
\multiput(1173.00,333.17)(1.500,1.000){2}{\rule{0.361pt}{0.400pt}}
\put(1159.0,334.0){\rule[-0.200pt]{3.373pt}{0.400pt}}
\put(1191,334.67){\rule{0.723pt}{0.400pt}}
\multiput(1191.00,334.17)(1.500,1.000){2}{\rule{0.361pt}{0.400pt}}
\put(1176.0,335.0){\rule[-0.200pt]{3.613pt}{0.400pt}}
\put(1206,335.67){\rule{0.482pt}{0.400pt}}
\multiput(1206.00,335.17)(1.000,1.000){2}{\rule{0.241pt}{0.400pt}}
\put(1194.0,336.0){\rule[-0.200pt]{2.891pt}{0.400pt}}
\put(1223,336.67){\rule{0.723pt}{0.400pt}}
\multiput(1223.00,336.17)(1.500,1.000){2}{\rule{0.361pt}{0.400pt}}
\put(1208.0,337.0){\rule[-0.200pt]{3.613pt}{0.400pt}}
\put(1241,337.67){\rule{0.482pt}{0.400pt}}
\multiput(1241.00,337.17)(1.000,1.000){2}{\rule{0.241pt}{0.400pt}}
\put(1226.0,338.0){\rule[-0.200pt]{3.613pt}{0.400pt}}
\put(1258,338.67){\rule{0.723pt}{0.400pt}}
\multiput(1258.00,338.17)(1.500,1.000){2}{\rule{0.361pt}{0.400pt}}
\put(1243.0,339.0){\rule[-0.200pt]{3.613pt}{0.400pt}}
\put(1276,339.67){\rule{0.482pt}{0.400pt}}
\multiput(1276.00,339.17)(1.000,1.000){2}{\rule{0.241pt}{0.400pt}}
\put(1261.0,340.0){\rule[-0.200pt]{3.613pt}{0.400pt}}
\put(1293,340.67){\rule{0.723pt}{0.400pt}}
\multiput(1293.00,340.17)(1.500,1.000){2}{\rule{0.361pt}{0.400pt}}
\put(1278.0,341.0){\rule[-0.200pt]{3.613pt}{0.400pt}}
\put(1311,341.67){\rule{0.723pt}{0.400pt}}
\multiput(1311.00,341.17)(1.500,1.000){2}{\rule{0.361pt}{0.400pt}}
\put(1296.0,342.0){\rule[-0.200pt]{3.613pt}{0.400pt}}
\put(1328,342.67){\rule{0.723pt}{0.400pt}}
\multiput(1328.00,342.17)(1.500,1.000){2}{\rule{0.361pt}{0.400pt}}
\put(1314.0,343.0){\rule[-0.200pt]{3.373pt}{0.400pt}}
\put(1346,343.67){\rule{0.723pt}{0.400pt}}
\multiput(1346.00,343.17)(1.500,1.000){2}{\rule{0.361pt}{0.400pt}}
\put(1331.0,344.0){\rule[-0.200pt]{3.613pt}{0.400pt}}
\put(1363,344.67){\rule{0.723pt}{0.400pt}}
\multiput(1363.00,344.17)(1.500,1.000){2}{\rule{0.361pt}{0.400pt}}
\put(1349.0,345.0){\rule[-0.200pt]{3.373pt}{0.400pt}}
\put(1381,345.67){\rule{0.723pt}{0.400pt}}
\multiput(1381.00,345.17)(1.500,1.000){2}{\rule{0.361pt}{0.400pt}}
\put(1366.0,346.0){\rule[-0.200pt]{3.613pt}{0.400pt}}
\put(1401,346.67){\rule{0.723pt}{0.400pt}}
\multiput(1401.00,346.17)(1.500,1.000){2}{\rule{0.361pt}{0.400pt}}
\put(1384.0,347.0){\rule[-0.200pt]{4.095pt}{0.400pt}}
\put(1419,347.67){\rule{0.723pt}{0.400pt}}
\multiput(1419.00,347.17)(1.500,1.000){2}{\rule{0.361pt}{0.400pt}}
\put(1404.0,348.0){\rule[-0.200pt]{3.613pt}{0.400pt}}
\put(1436,348.67){\rule{0.723pt}{0.400pt}}
\multiput(1436.00,348.17)(1.500,1.000){2}{\rule{0.361pt}{0.400pt}}
\put(1422.0,349.0){\rule[-0.200pt]{3.373pt}{0.400pt}}
\put(1454,349.67){\rule{0.723pt}{0.400pt}}
\multiput(1454.00,349.17)(1.500,1.000){2}{\rule{0.361pt}{0.400pt}}
\put(1439.0,350.0){\rule[-0.200pt]{3.613pt}{0.400pt}}
\put(1474,350.67){\rule{0.723pt}{0.400pt}}
\multiput(1474.00,350.17)(1.500,1.000){2}{\rule{0.361pt}{0.400pt}}
\put(1457.0,351.0){\rule[-0.200pt]{4.095pt}{0.400pt}}
\put(1492,351.67){\rule{0.723pt}{0.400pt}}
\multiput(1492.00,351.17)(1.500,1.000){2}{\rule{0.361pt}{0.400pt}}
\put(1477.0,352.0){\rule[-0.200pt]{3.613pt}{0.400pt}}
\put(1512,352.67){\rule{0.723pt}{0.400pt}}
\multiput(1512.00,352.17)(1.500,1.000){2}{\rule{0.361pt}{0.400pt}}
\put(1495.0,353.0){\rule[-0.200pt]{4.095pt}{0.400pt}}
\put(1530,353.67){\rule{0.723pt}{0.400pt}}
\multiput(1530.00,353.17)(1.500,1.000){2}{\rule{0.361pt}{0.400pt}}
\put(1515.0,354.0){\rule[-0.200pt]{3.613pt}{0.400pt}}
\put(1550,354.67){\rule{0.723pt}{0.400pt}}
\multiput(1550.00,354.17)(1.500,1.000){2}{\rule{0.361pt}{0.400pt}}
\put(1533.0,355.0){\rule[-0.200pt]{4.095pt}{0.400pt}}
\put(1568,355.67){\rule{0.482pt}{0.400pt}}
\multiput(1568.00,355.17)(1.000,1.000){2}{\rule{0.241pt}{0.400pt}}
\put(1553.0,356.0){\rule[-0.200pt]{3.613pt}{0.400pt}}
\put(1570.0,357.0){\rule[-0.200pt]{4.336pt}{0.400pt}}
\sbox{\plotpoint}{\rule[-0.400pt]{0.800pt}{0.800pt}}%
\sbox{\plotpoint}{\rule[-0.200pt]{0.400pt}{0.400pt}}%
\put(1428,418){\makebox(0,0)[r]{$\frac{d}{dx}\xi(x)$}}
\sbox{\plotpoint}{\rule[-0.400pt]{0.800pt}{0.800pt}}%
\put(1448.0,418.0){\rule[-0.400pt]{12.286pt}{0.800pt}}
\put(131,376){\usebox{\plotpoint}}
\put(163,373.84){\rule{0.723pt}{0.800pt}}
\multiput(163.00,374.34)(1.500,-1.000){2}{\rule{0.361pt}{0.800pt}}
\put(131.0,376.0){\rule[-0.400pt]{7.709pt}{0.800pt}}
\put(201,372.84){\rule{0.723pt}{0.800pt}}
\multiput(201.00,373.34)(1.500,-1.000){2}{\rule{0.361pt}{0.800pt}}
\put(166.0,375.0){\rule[-0.400pt]{8.431pt}{0.800pt}}
\put(239,371.84){\rule{0.723pt}{0.800pt}}
\multiput(239.00,372.34)(1.500,-1.000){2}{\rule{0.361pt}{0.800pt}}
\put(204.0,374.0){\rule[-0.400pt]{8.431pt}{0.800pt}}
\put(277,370.84){\rule{0.723pt}{0.800pt}}
\multiput(277.00,371.34)(1.500,-1.000){2}{\rule{0.361pt}{0.800pt}}
\put(242.0,373.0){\rule[-0.400pt]{8.431pt}{0.800pt}}
\put(315,369.84){\rule{0.723pt}{0.800pt}}
\multiput(315.00,370.34)(1.500,-1.000){2}{\rule{0.361pt}{0.800pt}}
\put(280.0,372.0){\rule[-0.400pt]{8.431pt}{0.800pt}}
\put(356,368.84){\rule{0.723pt}{0.800pt}}
\multiput(356.00,369.34)(1.500,-1.000){2}{\rule{0.361pt}{0.800pt}}
\put(318.0,371.0){\rule[-0.400pt]{9.154pt}{0.800pt}}
\put(394,367.84){\rule{0.723pt}{0.800pt}}
\multiput(394.00,368.34)(1.500,-1.000){2}{\rule{0.361pt}{0.800pt}}
\put(359.0,370.0){\rule[-0.400pt]{8.431pt}{0.800pt}}
\put(432,366.84){\rule{0.723pt}{0.800pt}}
\multiput(432.00,367.34)(1.500,-1.000){2}{\rule{0.361pt}{0.800pt}}
\put(397.0,369.0){\rule[-0.400pt]{8.431pt}{0.800pt}}
\put(470,365.84){\rule{0.723pt}{0.800pt}}
\multiput(470.00,366.34)(1.500,-1.000){2}{\rule{0.361pt}{0.800pt}}
\put(435.0,368.0){\rule[-0.400pt]{8.431pt}{0.800pt}}
\put(511,364.84){\rule{0.482pt}{0.800pt}}
\multiput(511.00,365.34)(1.000,-1.000){2}{\rule{0.241pt}{0.800pt}}
\put(473.0,367.0){\rule[-0.400pt]{9.154pt}{0.800pt}}
\put(549,363.84){\rule{0.482pt}{0.800pt}}
\multiput(549.00,364.34)(1.000,-1.000){2}{\rule{0.241pt}{0.800pt}}
\put(513.0,366.0){\rule[-0.400pt]{8.672pt}{0.800pt}}
\put(586,362.84){\rule{0.723pt}{0.800pt}}
\multiput(586.00,363.34)(1.500,-1.000){2}{\rule{0.361pt}{0.800pt}}
\put(551.0,365.0){\rule[-0.400pt]{8.431pt}{0.800pt}}
\put(624,361.84){\rule{0.723pt}{0.800pt}}
\multiput(624.00,362.34)(1.500,-1.000){2}{\rule{0.361pt}{0.800pt}}
\put(589.0,364.0){\rule[-0.400pt]{8.431pt}{0.800pt}}
\put(662,360.84){\rule{0.723pt}{0.800pt}}
\multiput(662.00,361.34)(1.500,-1.000){2}{\rule{0.361pt}{0.800pt}}
\put(627.0,363.0){\rule[-0.400pt]{8.431pt}{0.800pt}}
\put(703,359.84){\rule{0.723pt}{0.800pt}}
\multiput(703.00,360.34)(1.500,-1.000){2}{\rule{0.361pt}{0.800pt}}
\put(665.0,362.0){\rule[-0.400pt]{9.154pt}{0.800pt}}
\put(741,358.84){\rule{0.723pt}{0.800pt}}
\multiput(741.00,359.34)(1.500,-1.000){2}{\rule{0.361pt}{0.800pt}}
\put(706.0,361.0){\rule[-0.400pt]{8.431pt}{0.800pt}}
\put(779,357.84){\rule{0.723pt}{0.800pt}}
\multiput(779.00,358.34)(1.500,-1.000){2}{\rule{0.361pt}{0.800pt}}
\put(744.0,360.0){\rule[-0.400pt]{8.431pt}{0.800pt}}
\put(817,356.84){\rule{0.723pt}{0.800pt}}
\multiput(817.00,357.34)(1.500,-1.000){2}{\rule{0.361pt}{0.800pt}}
\put(782.0,359.0){\rule[-0.400pt]{8.431pt}{0.800pt}}
\put(855,355.84){\rule{0.723pt}{0.800pt}}
\multiput(855.00,356.34)(1.500,-1.000){2}{\rule{0.361pt}{0.800pt}}
\put(820.0,358.0){\rule[-0.400pt]{8.431pt}{0.800pt}}
\put(896,354.84){\rule{0.723pt}{0.800pt}}
\multiput(896.00,355.34)(1.500,-1.000){2}{\rule{0.361pt}{0.800pt}}
\put(858.0,357.0){\rule[-0.400pt]{9.154pt}{0.800pt}}
\put(934,353.84){\rule{0.723pt}{0.800pt}}
\multiput(934.00,354.34)(1.500,-1.000){2}{\rule{0.361pt}{0.800pt}}
\put(899.0,356.0){\rule[-0.400pt]{8.431pt}{0.800pt}}
\put(972,352.84){\rule{0.723pt}{0.800pt}}
\multiput(972.00,353.34)(1.500,-1.000){2}{\rule{0.361pt}{0.800pt}}
\put(937.0,355.0){\rule[-0.400pt]{8.431pt}{0.800pt}}
\put(1010,351.84){\rule{0.723pt}{0.800pt}}
\multiput(1010.00,352.34)(1.500,-1.000){2}{\rule{0.361pt}{0.800pt}}
\put(975.0,354.0){\rule[-0.400pt]{8.431pt}{0.800pt}}
\put(1051,350.84){\rule{0.723pt}{0.800pt}}
\multiput(1051.00,351.34)(1.500,-1.000){2}{\rule{0.361pt}{0.800pt}}
\put(1013.0,353.0){\rule[-0.400pt]{9.154pt}{0.800pt}}
\put(1089,349.84){\rule{0.723pt}{0.800pt}}
\multiput(1089.00,350.34)(1.500,-1.000){2}{\rule{0.361pt}{0.800pt}}
\put(1054.0,352.0){\rule[-0.400pt]{8.431pt}{0.800pt}}
\put(1127,348.84){\rule{0.723pt}{0.800pt}}
\multiput(1127.00,349.34)(1.500,-1.000){2}{\rule{0.361pt}{0.800pt}}
\put(1092.0,351.0){\rule[-0.400pt]{8.431pt}{0.800pt}}
\put(1165,347.84){\rule{0.723pt}{0.800pt}}
\multiput(1165.00,348.34)(1.500,-1.000){2}{\rule{0.361pt}{0.800pt}}
\put(1130.0,350.0){\rule[-0.400pt]{8.431pt}{0.800pt}}
\put(1203,346.84){\rule{0.723pt}{0.800pt}}
\multiput(1203.00,347.34)(1.500,-1.000){2}{\rule{0.361pt}{0.800pt}}
\put(1168.0,349.0){\rule[-0.400pt]{8.431pt}{0.800pt}}
\put(1243,345.84){\rule{0.723pt}{0.800pt}}
\multiput(1243.00,346.34)(1.500,-1.000){2}{\rule{0.361pt}{0.800pt}}
\put(1206.0,348.0){\rule[-0.400pt]{8.913pt}{0.800pt}}
\put(1281,344.84){\rule{0.723pt}{0.800pt}}
\multiput(1281.00,345.34)(1.500,-1.000){2}{\rule{0.361pt}{0.800pt}}
\put(1246.0,347.0){\rule[-0.400pt]{8.431pt}{0.800pt}}
\put(1319,343.84){\rule{0.723pt}{0.800pt}}
\multiput(1319.00,344.34)(1.500,-1.000){2}{\rule{0.361pt}{0.800pt}}
\put(1284.0,346.0){\rule[-0.400pt]{8.431pt}{0.800pt}}
\put(1357,342.84){\rule{0.723pt}{0.800pt}}
\multiput(1357.00,343.34)(1.500,-1.000){2}{\rule{0.361pt}{0.800pt}}
\put(1322.0,345.0){\rule[-0.400pt]{8.431pt}{0.800pt}}
\put(1398,341.84){\rule{0.723pt}{0.800pt}}
\multiput(1398.00,342.34)(1.500,-1.000){2}{\rule{0.361pt}{0.800pt}}
\put(1360.0,344.0){\rule[-0.400pt]{9.154pt}{0.800pt}}
\put(1436,340.84){\rule{0.723pt}{0.800pt}}
\multiput(1436.00,341.34)(1.500,-1.000){2}{\rule{0.361pt}{0.800pt}}
\put(1401.0,343.0){\rule[-0.400pt]{8.431pt}{0.800pt}}
\put(1474,339.84){\rule{0.723pt}{0.800pt}}
\multiput(1474.00,340.34)(1.500,-1.000){2}{\rule{0.361pt}{0.800pt}}
\put(1439.0,342.0){\rule[-0.400pt]{8.431pt}{0.800pt}}
\put(1512,338.84){\rule{0.723pt}{0.800pt}}
\multiput(1512.00,339.34)(1.500,-1.000){2}{\rule{0.361pt}{0.800pt}}
\put(1477.0,341.0){\rule[-0.400pt]{8.431pt}{0.800pt}}
\put(1550,337.84){\rule{0.723pt}{0.800pt}}
\multiput(1550.00,338.34)(1.500,-1.000){2}{\rule{0.361pt}{0.800pt}}
\put(1515.0,340.0){\rule[-0.400pt]{8.431pt}{0.800pt}}
\put(1553.0,339.0){\rule[-0.400pt]{8.431pt}{0.800pt}}
\sbox{\plotpoint}{\rule[-0.200pt]{0.400pt}{0.400pt}}%
\put(131.0,131.0){\rule[-0.200pt]{0.400pt}{88.651pt}}
\put(131.0,131.0){\rule[-0.200pt]{350.991pt}{0.400pt}}
\put(1588.0,131.0){\rule[-0.200pt]{0.400pt}{88.651pt}}
\put(131.0,499.0){\rule[-0.200pt]{350.991pt}{0.400pt}}
\end{picture}
\begin{picture}(1650,540)(0,0)
\put(131.0,169.0){\rule[-0.200pt]{4.818pt}{0.400pt}}
\put(111,169){\makebox(0,0)[r]{-1}}
\put(1568.0,169.0){\rule[-0.200pt]{4.818pt}{0.400pt}}
\put(131.0,263.0){\rule[-0.200pt]{4.818pt}{0.400pt}}
\put(111,263){\makebox(0,0)[r]{ 0}}
\put(1568.0,263.0){\rule[-0.200pt]{4.818pt}{0.400pt}}
\put(131.0,312.0){\rule[-0.200pt]{4.818pt}{0.400pt}}
\put(111,312){\makebox(0,0)[r]{$\beta$}}
\put(1568.0,312.0){\rule[-0.200pt]{4.818pt}{0.400pt}}
\put(131.0,357.0){\rule[-0.200pt]{4.818pt}{0.400pt}}
\put(111,357){\makebox(0,0)[r]{ 1}}
\put(1568.0,357.0){\rule[-0.200pt]{4.818pt}{0.400pt}}
\put(131.0,452.0){\rule[-0.200pt]{4.818pt}{0.400pt}}
\put(111,452){\makebox(0,0)[r]{ 2}}
\put(1568.0,452.0){\rule[-0.200pt]{4.818pt}{0.400pt}}
\put(131.0,131.0){\rule[-0.200pt]{0.400pt}{4.818pt}}
\put(131,90){\makebox(0,0){0}}
\put(131.0,479.0){\rule[-0.200pt]{0.400pt}{4.818pt}}
\put(495.0,131.0){\rule[-0.200pt]{0.400pt}{4.818pt}}
\put(495,90){\makebox(0,0){0.25}}
\put(495.0,479.0){\rule[-0.200pt]{0.400pt}{4.818pt}}
\put(860.0,131.0){\rule[-0.200pt]{0.400pt}{4.818pt}}
\put(860,90){\makebox(0,0){$\alpha$}}
\put(860.0,479.0){\rule[-0.200pt]{0.400pt}{4.818pt}}
\put(1224.0,131.0){\rule[-0.200pt]{0.400pt}{4.818pt}}
\put(1224,90){\makebox(0,0){0.75}}
\put(1224.0,479.0){\rule[-0.200pt]{0.400pt}{4.818pt}}
\put(1588.0,131.0){\rule[-0.200pt]{0.400pt}{4.818pt}}
\put(1588.0,479.0){\rule[-0.200pt]{0.400pt}{4.818pt}}
\put(131.0,131.0){\rule[-0.200pt]{0.400pt}{88.651pt}}
\put(131.0,131.0){\rule[-0.200pt]{350.991pt}{0.400pt}}
\put(1588.0,131.0){\rule[-0.200pt]{0.400pt}{88.651pt}}
\put(131.0,499.0){\rule[-0.200pt]{350.991pt}{0.400pt}}
\put(30,315){\makebox(0,0){$y$}}
\put(859,29){\makebox(0,0){$x$}}
\put(1428,459){\makebox(0,0)[r]{$w(x)$}}
\put(1448.0,459.0){\rule[-0.200pt]{12.286pt}{0.400pt}}
\put(131,169){\usebox{\plotpoint}}
\multiput(858.61,169.00)(0.447,41.765){3}{\rule{0.108pt}{25.167pt}}
\multiput(857.17,169.00)(3.000,135.765){2}{\rule{0.400pt}{12.583pt}}
\put(131.0,169.0){\rule[-0.200pt]{175.134pt}{0.400pt}}
\put(861.0,357.0){\rule[-0.200pt]{175.134pt}{0.400pt}}
\sbox{\plotpoint}{\rule[-0.500pt]{1.000pt}{1.000pt}}%
\sbox{\plotpoint}{\rule[-0.200pt]{0.400pt}{0.400pt}}%
\put(1428,418){\makebox(0,0)[r]{$v(\xi(x))$}}
\sbox{\plotpoint}{\rule[-0.500pt]{1.000pt}{1.000pt}}%
\multiput(1448,418)(20.756,0.000){3}{\usebox{\plotpoint}}
\put(1499,418){\usebox{\plotpoint}}
\put(131,169){\usebox{\plotpoint}}
\put(131.00,169.00){\usebox{\plotpoint}}
\put(151.76,169.00){\usebox{\plotpoint}}
\put(172.35,170.00){\usebox{\plotpoint}}
\put(192.94,171.00){\usebox{\plotpoint}}
\put(213.70,171.00){\usebox{\plotpoint}}
\put(234.29,172.00){\usebox{\plotpoint}}
\put(255.05,172.00){\usebox{\plotpoint}}
\put(275.64,173.00){\usebox{\plotpoint}}
\put(296.39,173.00){\usebox{\plotpoint}}
\put(316.99,174.00){\usebox{\plotpoint}}
\put(337.74,174.00){\usebox{\plotpoint}}
\put(358.50,174.00){\usebox{\plotpoint}}
\put(379.09,175.00){\usebox{\plotpoint}}
\put(399.85,175.00){\usebox{\plotpoint}}
\put(420.44,176.00){\usebox{\plotpoint}}
\put(441.20,176.00){\usebox{\plotpoint}}
\put(461.95,176.00){\usebox{\plotpoint}}
\put(482.55,177.00){\usebox{\plotpoint}}
\put(503.30,177.00){\usebox{\plotpoint}}
\put(524.06,177.00){\usebox{\plotpoint}}
\put(544.65,178.00){\usebox{\plotpoint}}
\put(565.40,178.00){\usebox{\plotpoint}}
\put(586.16,178.00){\usebox{\plotpoint}}
\put(606.92,178.00){\usebox{\plotpoint}}
\put(627.51,179.00){\usebox{\plotpoint}}
\put(648.26,179.00){\usebox{\plotpoint}}
\put(669.02,179.00){\usebox{\plotpoint}}
\put(689.78,179.00){\usebox{\plotpoint}}
\put(710.37,180.00){\usebox{\plotpoint}}
\put(731.12,180.00){\usebox{\plotpoint}}
\put(751.88,180.00){\usebox{\plotpoint}}
\put(772.64,180.00){\usebox{\plotpoint}}
\put(793.39,180.00){\usebox{\plotpoint}}
\put(814.15,180.00){\usebox{\plotpoint}}
\put(834.74,181.00){\usebox{\plotpoint}}
\put(855.49,181.00){\usebox{\plotpoint}}
\multiput(858,181)(0.360,20.752){8}{\usebox{\plotpoint}}
\put(872.26,354.00){\usebox{\plotpoint}}
\put(892.86,353.00){\usebox{\plotpoint}}
\put(913.61,353.00){\usebox{\plotpoint}}
\put(934.13,352.00){\usebox{\plotpoint}}
\put(954.73,351.00){\usebox{\plotpoint}}
\put(975.48,351.00){\usebox{\plotpoint}}
\put(996.07,350.00){\usebox{\plotpoint}}
\put(1016.83,350.00){\usebox{\plotpoint}}
\put(1037.35,349.00){\usebox{\plotpoint}}
\put(1058.11,349.00){\usebox{\plotpoint}}
\put(1078.86,349.00){\usebox{\plotpoint}}
\put(1099.45,348.00){\usebox{\plotpoint}}
\put(1120.21,348.00){\usebox{\plotpoint}}
\put(1140.97,348.00){\usebox{\plotpoint}}
\put(1161.72,348.00){\usebox{\plotpoint}}
\put(1182.48,348.00){\usebox{\plotpoint}}
\put(1203.23,348.00){\usebox{\plotpoint}}
\put(1223.99,348.00){\usebox{\plotpoint}}
\put(1244.65,348.55){\usebox{\plotpoint}}
\put(1265.34,349.00){\usebox{\plotpoint}}
\put(1286.09,349.00){\usebox{\plotpoint}}
\put(1306.85,349.00){\usebox{\plotpoint}}
\put(1327.44,350.00){\usebox{\plotpoint}}
\put(1348.20,350.00){\usebox{\plotpoint}}
\put(1368.79,351.00){\usebox{\plotpoint}}
\put(1389.54,351.00){\usebox{\plotpoint}}
\put(1410.14,352.00){\usebox{\plotpoint}}
\put(1430.85,352.28){\usebox{\plotpoint}}
\put(1451.49,353.00){\usebox{\plotpoint}}
\put(1472.08,354.00){\usebox{\plotpoint}}
\put(1492.84,354.00){\usebox{\plotpoint}}
\put(1513.35,355.00){\usebox{\plotpoint}}
\put(1533.95,356.00){\usebox{\plotpoint}}
\put(1554.70,356.00){\usebox{\plotpoint}}
\put(1575.30,357.00){\usebox{\plotpoint}}
\put(1588,357){\usebox{\plotpoint}}
\sbox{\plotpoint}{\rule[-0.200pt]{0.400pt}{0.400pt}}%
\put(131.0,131.0){\rule[-0.200pt]{0.400pt}{88.651pt}}
\put(131.0,131.0){\rule[-0.200pt]{350.991pt}{0.400pt}}
\put(1588.0,131.0){\rule[-0.200pt]{0.400pt}{88.651pt}}
\put(131.0,499.0){\rule[-0.200pt]{350.991pt}{0.400pt}}
\end{picture}

\caption{$v$ approximates $w$ sufficiently close in the sense of definition \ref{sufficientlyclose}. First plot: The given functions $w$ and $v$. Second plot: The coordinate transformation $\xi$ and its derivative. Third plot: $v(\xi(x))$ approximates $w(x)$ pointwise. In this picture, we chose $\xi(x) := \frac{\beta}{\alpha}x + x (x-\alpha) \frac{(1-{\beta}/{\alpha})}{1-\alpha}$.}
\label{fig:suff_close}
\end{figure}
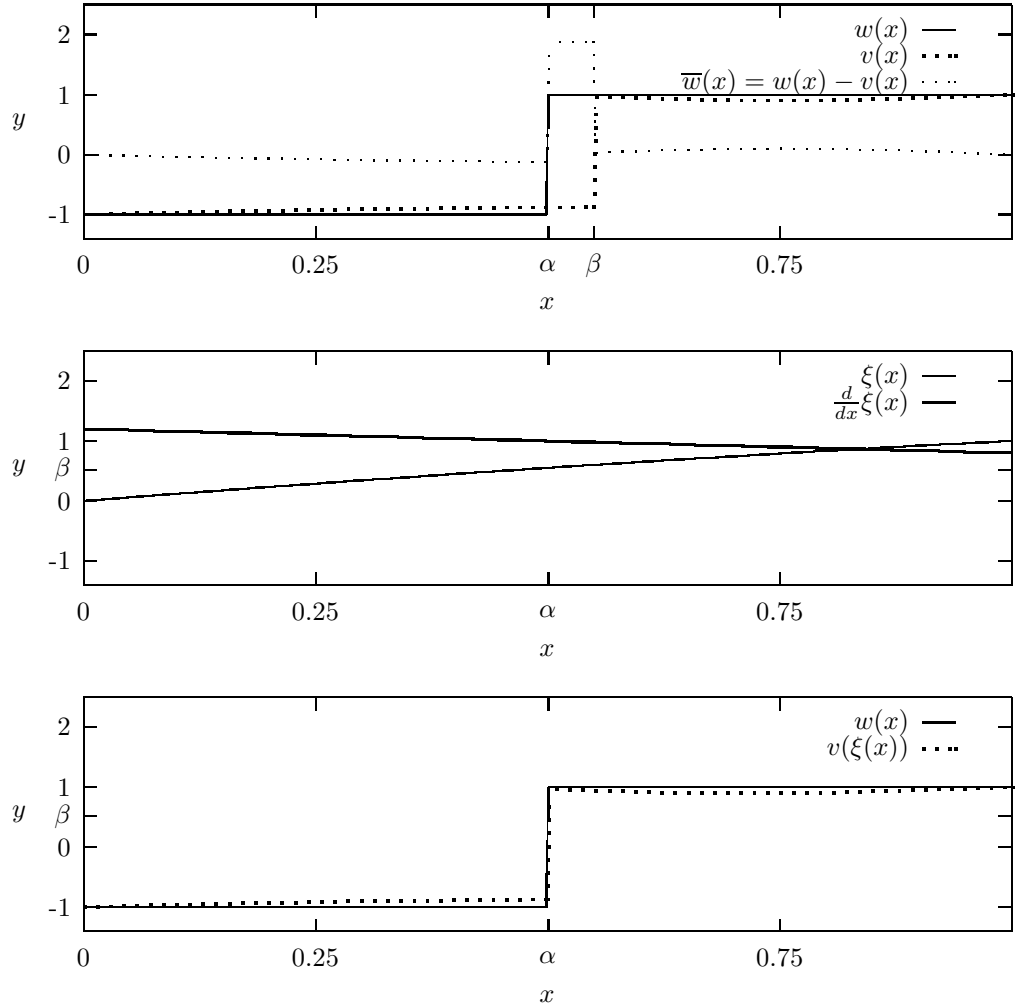

\begin{lemma}
    Let $\xi_1$ and $\xi_2$ be smooth, invertible functions fulfilling the properties \eqref{approx_xistrich1}, \eqref{approx_xistrich2} and \eqref{approx_xistrich3} with domains as given in \eqref{def:xi1} and \eqref{def:xi2}. 
    Then 
    \begin{align} 
      \label{xi_x_Onu}
      \xi_i(x) - x = O(\nu) \quad \forall i = 1,2, 
    \end{align} 
  which, as a special case, implies $\beta - \alpha = O(\nu)$. 
\end{lemma}
\begin{proof} 
  (We consider only the case $i = 1$. The case $i = 2$ is completely analogous with the obvious interchange of $0$ and $1$.) We have that
  \begin{alignat*}{2}
    \xi_1(0) - 0 &= 0  &\quad& \text{due to the invertibility of $\xi_1$ and } \\ 
    \frac{d}{dx}(\xi_1(x) - x) &= O(\nu) &\quad& \text{due to (\ref{approx_xistrich1}).}
  \end{alignat*}
This proves the claim because we can write
  \begin{align*}
   \xi_1(x) - x = \xi_1(0) - 0 + \int_0^x \frac{d}{dx} \left( \xi_1(\tau) - \tau \right) \ d \tau = O(\nu).
  \end{align*}

\end{proof}
\begin{lemma}
 Let $w$ be a piecewise smooth function with a jump in $x=\alpha$, and let $v$ be sufficiently close to $w$ in the sense of definition \ref{sufficientlyclose}. Then
    \begin{align} 
	\label{eq:x_close}
	v(x) - w(x) &= O(\nu) \quad \forall x \in \Omega \backslash [\alpha,\beta].
    \end{align}
\end{lemma}
\begin{proof}
(Without loss of generality, $x < \alpha $)
\begin{align*} 
				v(x) - w(x) &\stackrel{\hphantom{(8.4)}}{=} v(\xi_1(x)) - w(x) + v(x) - v(\xi_1(x)) \\ 
				&\stackrel{(\ref{def:sufficiently_smooth})}=O(\nu) + v(x) - v(\xi_1(x)) \\
				&\stackrel{\hphantom{(8.4)}}= O(\nu) + v'(\xi_1(x)) \cdot (x - \xi_1(x)) + O(\|x-\xi_1(x)\|^2)
				\stackrel{(\ref{xi_x_Onu})}= O(\nu).
\end{align*}
\end{proof}

\subsection{Linearization of the Rankine-Hugoniot Condition}
Not every discontinuity of $w$ is permissible. A very basic restriction following directly from the weak formulation of a hyperbolic conservation law is the Rankine-Hugoniot condition, which states that the flux has a (weak) divergence, even in the vicinity of a shock, more precisely,  
\begin{align} 
  \label{rh_intbc}
  [f(w)] &:= f(w(\alpha))^+ - f(w(\alpha))^- = 0. 
\end{align}  
Let us again assume that we are interested in a perturbed solution $v = w + \overline w$ (in the sense of definition \ref{defsufficientlysmooth}) which has its only shock at $x = \beta = \alpha + \overline \alpha$. In this section, we investigate how the Rankine-Hugoniot condition changes for such a $v$.

Let us first state the following lemma: 
\begin{lemma}
 Let $w$ be a piecewise smooth function with a jump in $x=\alpha$, and let $v$ be sufficiently close to $w$ in the sense of definition \ref{sufficientlyclose}. Furthermore, let $f \equiv f(w)$ be a smooth function. Then 
\begin{align} 
  \label{linearization_of_the_jump_in_f}
  [f(v)] =  & [f(w)] + f'(w(\beta)) (v(\beta)-w(\beta))^+ -  f'(w(\alpha))^- (v(\alpha)-w(\alpha))^-   \\ \notag & + \overline \alpha [\frac{d}{dx} f(w(x))] + o(\nu).
\end{align} 
Note that $[f(v)]$ denotes a jump at $x=\beta$, while $[f(w)]$ and $[\frac{d}{dx} f(w(x))]$ denote jumps at $x = \alpha$.
\end{lemma}
\begin{proof}
 We have that 
 \begin{align*} 
			f(v(\beta))^- &= f(w(\alpha))^- + f(v(\beta))^- - f(w(\alpha))^- \\
			&= f(w(\alpha))^- + f(v(\alpha))^- - f(w(\alpha))^- + f(v(\beta))^- - f(v(\alpha))^- \\
			&= f(w(\alpha))^- + f'(w(\alpha))^- (v(\alpha)-w(\alpha))^- + \frac{d (f \circ v)}{dx} (\alpha)^- \cdot \overline \alpha + O(\nu^2) \\
			&= f(w(\alpha))^- + f'(w(\alpha))^- (v(\alpha)-w(\alpha))^- + \frac{d (f \circ w)}{dx} (\alpha)^- \cdot \overline \alpha + o(\nu) .
\end{align*}
The last step is true because by replacing $\frac{d (f \circ v)}{dx}$ by $\frac{d (f \circ w)}{dx}$ we make an $O(\mu)$ error which is augmented to $O(\mu \nu) = o(\nu)$ by multiplying it with $\overline \alpha$. 
By treating $f(v(\beta))^+$ in an analog manner, and then subtracting $f(v(\beta))^-$ from $f(v(\beta))^+$, we get the claimed identity (\ref{linearization_of_the_jump_in_f}). 
\end{proof}

Let us, for the ease of notation, define 
\begin{align}
 \label{definition_strange_jump}
 [f'(w) \overline w] &:= f'(w(\beta)) (v(\beta)-w(\beta))^+ -  f'(w(\alpha))^- (v(\alpha)-w(\alpha))^-.
\end{align}
Inserting \eqref{definition_strange_jump} into \eqref{linearization_of_the_jump_in_f}, exploiting the Rankine-Hugoniot condition as given in (\ref{rh_intbc}), and assuming that $w$ solves (\ref{eq:conservation_law}), the jump in $f(v)$ can be linearized as
\begin{align}
 \label{lin_F}
 [f(v)] = [f'(w) \overline w] + \overline \alpha [\frac{d}{dx} f(w(x))] + o(\nu).
\end{align}

\subsection{Linearization of the Functional}
We are interested in computing the changes in the functional 
\begin{align} 
\label{intbcfunct}
\JJ(w) := \int_{\Omega} p(w) \dx,  
\end{align}
 with $p$ being sufficiently regular. Again, we assume that $v$ is sufficiently close to $w$. We can then compute
\begin{align*} 
				\JJ(v) - \JJ(w) &\stackrel{\hphantom{(\ref{eq:x_close})}}= \int_{\Omega} p(v) - p(w) \dx \\
				&\stackrel{\hphantom{(\ref{eq:x_close})}}= \int_{\Omega \backslash [\alpha, \beta]} p'(w) (v - w) \dx + O(\nu^2) + \int_{\alpha}^{\beta} p(v) - p(w) \dx \\
				&\stackrel{\hphantom{(\ref{eq:x_close})}}= \int_{\Omega \backslash [\alpha, \beta]} p'(w) (v - w) \dx + O(\nu^2) + (\beta - \alpha) (p(v(\alpha)) - p(w(\alpha))^+) \\ &\hphantom{\stackrel{\hphantom{(\ref{eq:x_close})}}=}+ O(\nu^2) \\
				&\stackrel{(\ref{eq:x_close})}= \int_{\Omega \backslash [\alpha, \beta]} p'(w) (v - w) \dx + \overline \alpha (p(w(\alpha))^- - p(w(\alpha))^+) + O(\nu^2) \\
				&\stackrel{\hphantom{(\ref{eq:x_close})}}= \int_{\Omega \backslash [\alpha, \beta]} p'(w) (v - w) \dx - \overline \alpha [p(w)] + O(\nu^2) ,
\end{align*}
and in summary, we have the following lemma:
\begin{lemma}
Let $w$ be a piecewise smooth function with a jump in $x=\alpha$, and let $v$ be sufficiently close to $w$ in the sense of definition \ref{sufficientlyclose}. Let $\JJ$ be given by (\ref{intbcfunct}). Then 
\begin{align}
 \label{lin_J}
 \JJ(v) - \JJ(w) = \int_{\Omega \backslash [\alpha, \beta]} p'(w) (v-w) \dx - \overline \alpha [p(w)] + O(\nu^2).
\end{align}

\end{lemma}

\subsection{Adjoint Approach}
In this section, we put together the information from the previous subsections, and show that the adjoint error control works under suitable assumptions as usual. 
We make the following \emph{consistent} modification to the functional $\JJ$ and consider
\begin{align}
 \label{modified_augmented_int}
 J(w) &:= \JJ(w) - z_{\alpha}^T [f(w)] 
\end{align}
instead of $\JJ$ as in (\ref{intbcfunct}). $z_{\alpha} \in \R^d$ is a parameter that will be determined later. The modification is consistent, as $\JJ(w) = J(w)$ for a solution $w$ to (\ref{eq:conservation_law}). The latter is due to the fact that $[f(w)]$ vanishes. The same modification has already been done in \cite{GP00}. 

We assume that the dual solution $z$ is given as in (\ref{eq:adjoint}), and we additionally assume that it is at least Lipschitz-continuous. This is in good agreement with both our numerical experiences and Tadmor's theory for scalar conservation laws proposed in \cite{Tadmor}. In this section, we do not care about boundary conditions at all, as the focus is just on the behavior of the adjoint in the shock. We thus assume that all terms occurring at the (physical) boundary vanish, more precisely, 
\begin{align}
 \label{eq:cond_bdry}
  v(0) - w(0) = v(1) - w(1) = 0.
\end{align}

Putting all our information together, we can state the following theorem:
\begin{satz}\label{satz_adjoint_mit_bc}
 Let $w$ be a piecewise smooth, exact solution to (\ref{eq:conservation_law}) with a jump at $x = \alpha$, and $v$ be an approximation to $w$ in the sense of definition \ref{defsufficientlysmooth}, for which additionally holds $v = w$ at the boundary, i.e., $\overline w = v-w$ vanishes at $x = 0$ and $x = 1$. Furthermore, let $z$ be a smooth (at least Lipschitz-continuous) solution to (\ref{eq:adjoint}).
 The functional $J$ is defined as in (\ref{modified_augmented_int}) for a sufficiently smooth function $p \equiv p(w)$. Upon choosing $z_{\alpha} := z(\alpha)$, we can write 
\begin{align}
\label{fehlerdarstellungmitinternerrb}
 J(v) - J(w) = &\int_{\Omega} z^T (f(v)_x + S(v)) \ dx + \overline \alpha \left(- z^T(\alpha) [\frac{d}{dx} f(w) ] - [p(w)] \right) + o(\nu). 
\end{align}
\end{satz}

\begin{proof} The proof is a direct computation, it exploits the already known linearizations of both $[f(w)]$ and $\JJ(w)$:

\begin{alignat*}{4}
 \JJ(v) - \JJ(w) 
    &\overset{\eqref{lin_J}}=
     &&\int_{\Omega \backslash [\alpha, \beta]} p'(w) (v - w) \ dx 
    -\overline \alpha [p(w)] + o(\nu) \\ 
\displaybreak[0]
    &\centering \underset{\hphantom{\eqref{lin_J}}}{\stackrel{\eqref{eq:adjoint}}=}
    &&\int_{\Omega \backslash [\alpha, \beta]} (-f'(w)^T z_x + S'(w)^T z) (v - w) \ dx 
    -\overline \alpha [p(w)] +o(\nu) \\
\displaybreak[0]
    &\stackrel{\hphantom{\eqref{lin_J}}}= 
    && \int_{\Omega \backslash [\alpha, \beta]} (z^T_x \left( f(w) - f(v) \right) 
    + z^T \left( S(v) - S(w) \right) \dx -\overline \alpha [p(w)] +o(\nu) \\
\displaybreak[0]
    &\underset{\hphantom{\eqref{lin_J}}}{\stackrel{(\ref{eq:conservation_law})}=} &&\int_{\Omega \backslash [\alpha, \beta]} z^T (f(v)_x + S(v)) \ dx \\
    &&& - z^T(\alpha) f'(w(\alpha))(v(\alpha) - w(\alpha))^- 
    + z^T(\beta) f'(w(\beta))(v(\beta) - w(\beta))^+ \\
    &&&-\overline \alpha [p(w)] + o(\nu) \\
\displaybreak[0]
    & \underset{\hphantom{\eqref{lin_J}}}{\stackrel{(\ref{definition_strange_jump})}=} &&\int_{\Omega \backslash [\alpha, \beta]} z^T (f(v)_x + S(v)) \ dx + z^T(\alpha) [f'(w)\overline w]-\overline \alpha [p(w)] + o(\nu),
\end{alignat*}
where the last step is allowed due to the assumed Lipschitz-continuity of $z$, i.e., $z(\beta) = z(\alpha) + O(\nu)$;  and the fact that $[f'(w) \overline w]$ is of order $\nu$. 

Based on this computation, we can conclude that 
\begin{align}
 J(v) - J(w) \overset{\eqref{lin_F}}= &\int_{\Omega} z^T (f(v)_x + S(v)) \dx - \int_{\alpha}^{\beta} z^T (f(v)_x + S(v)) \dx + z^T(\alpha) [f'(w)\overline w] \\
	      \notag
	       &-z_{\alpha}^T \left( [f'(w) \overline w] + \overline \alpha [\frac{d}{dx}f(w)] \right) -\overline \alpha [p(w)] + o(\nu) \\
	     \stackrel{\hphantom{\eqref{lin_F}}}= &\int_{\Omega} z^T (f(v)_x + S(v)) \dx  + z^T(\alpha) [f'(w)\overline w] \\
	      \notag
	       &-z_{\alpha}^T \left( [f'(w) \overline w] + \overline \alpha [\frac{d}{dx}f(w)] \right) -\overline \alpha [p(w)] + o(\nu) + O(\nu \mu),
\end{align}
where the last step is true because $f(v)_x + S(v)$ is of order $\mu$.
Now upon choosing $z_{\alpha} := z(\alpha)$, we proved our claim (\ref{fehlerdarstellungmitinternerrb}) as the terms involving $[f'(w) \overline w]$ cancel each other. 
\end{proof}

Let us now make the following definition of what we mean by \emph{interior boundary condition}: 
\begin{definition}\label{def:ib}\index{Interior Boundary Condition}
 A function $z$ fulfills the \emph{interior boundary condition} with respect to $p$ and the shock position $\alpha$, iff
 \begin{align}
  \label{dual_boundary_condition}
  z^T(\alpha) [f(w)_x] = -[p(w)]
 \end{align}
 for the solution $w$ to \eqref{eq:conservation_law}. 
\end{definition}

In the next section (section \ref{sec:convergence_intbc}), we prove that the solution $z$ to \eqref{eq:adjoint}, under standard assumptions, fulfills (\ref{dual_boundary_condition}).

\begin{korollar}
 Suppose that the adjoint solution $z$ as given in (\ref{eq:adjoint}) fulfills the interior boundary condition (\ref{dual_boundary_condition}), we have under the assumptions of Theorem \ref{satz_adjoint_mit_bc} the \emph{usual} adjoint error representation
\begin{align*}
 J(v) - J(w) = &\int_{\Omega} z^T (f(v)_x + S(v)) \ dx + o(\nu).
\end{align*}
\end{korollar}

\subsection{Interior Boundary Condition for the Euler Equations}
A prototype of (\ref{eq:conservation_law}) with $d = 3$ are the steady-state quasi one-dimensional Euler equations, which are a model for compressible nozzle flow. 
$w$, $f$ and $S$ are defined as 
\begin{align} \notag
	      w &= (\rho, \rho u, E)^T, \\
	      \label{eq:eulerfluxes}
	      f(w) &= (\rho u, \rho u^2 + p, u(E+p))^T, \\ \notag
	      S(w) &= \frac{A'}{A} (\rho u, \rho u^2, u(E+p))^T,
\end{align}
respectively. 
The conservative variables $\rho, \rho u, E$ are density, momentum (which equals density times velocity) and total energy. Furthermore, $A \equiv A(x)$ describes the nozzle geometry (assumed to be rotational-symmetric, so $A(x)$ does in fact describe the diameter) and
\begin{align} 
  \label{pressure} p(w) := (\gamma-1)(E-\frac{1}{2}  \rho u^2) 
\end{align} 
is the pressure, where we have used a specific equation of state for $p$ that holds for a polytropic ideal gas, and $\gamma$ is the ratio of specific heats,  a gas-specific constant, which takes $\gamma = 1.4$ for an ideal di-atomic gas, of which air is a specific example.

Boundary conditions $\Ug$ can, for example, be set as
\begin{align} 
  \label{qee_bdry}
  0 = \Ug(w):= \begin{cases} p - p_0 & \text{on the outflow boundary}  \\ 
		  (s,h) - (s_0,h_0) & \text{on the inflow boundary} \end{cases}
\end{align} 
where $s = \alpha_0 \log(\frac{p}{\rho^{\gamma}}) + \alpha_1$ denotes entropy and $h = \frac{c^2}{\gamma - 1} + \frac{u^2}{2}$ total enthalpy,  i.e., one prescribes the pressure $p_0$ at the outflow, enthalpy $h_0$ and entropy $s_0$ at the inflow. Here $\alpha_i$ are constants, and $c$ denotes the speed of sound. 

We consider the target functional as in \eqref{intbcfunct}, where $p$ denotes pressure as defined in \eqref{pressure}.
For the Euler equations, i.e., eq. (\ref{eq:conservation_law}) with $f$ and $S$ defined as in (\ref{eq:eulerfluxes}), one can make (\ref{dual_boundary_condition}) more explicit as follows (of course, this has already been done, for example in \cite{GP97}): 
Due to the underlying equation (\ref{eq:conservation_law}), we have for $w = (w_1, w_2, w_3) = (\rho, \rho u, E)$
\begin{align}
      [f(w)_x] &= -[S(w)], \\
      [S(w)] &= \frac{A'(\alpha)}{A(\alpha)} ([\rho u], [\rho u^2], [u(E+p(w))])
\intertext{and due to Rankine-Hugoniot, we have that }
      [f(w)] &= ([\rho u], [\rho u^2 + p(w)], [u(E+p(w))]) = 0 
\intertext{which yields }
      [S(w)] &= \frac{A'(\alpha)}{A(\alpha)} (0, [\rho u^2], 0). 
\intertext{Substituting all this information into the interior boundary condition (\ref{dual_boundary_condition}), we get for $z = (z_1, z_2, z_3)^T$}
      z_2(\alpha) \frac{-A'(\alpha)}{A(\alpha)} [\rho u^2] &= -[p(w)],
\intertext{which yields}
      z_2(\alpha) &= \frac{A(\alpha)}{A'(\alpha)}\frac{[p(w)]}{[\rho u^2]}. 
\intertext{Again, thanks to Rankine-Hugoniot, we have }
      [\rho u^2] &= -[p(w)],
\intertext{which in all yields the internal adjoint boundary condition for the Euler equations, }
\label{internal_boundary} z_2(\alpha) &= - \frac{A(\alpha)}{A'(\alpha)} .
\end{align} 

Usually, (\ref{internal_boundary}) is of course not enforced in a numerical procedure, as for example $\alpha$ is in general not known. Due to the fact that numerical schemes in general approximate the solution $w$ by a viscous regularization, it has been argued that neglecting (\ref{internal_boundary}) is reasonable, because in the vanishing viscosity limit, $z$ is supposed to fulfill (\ref{internal_boundary}). This, however, has to our knowledge not been proven. 
In the following section, we therefore show that, assuming the adjoint solution can be seen as a limit of a viscous adjoint (to be defined below), the exact adjoint fulfills the interior boundary condition.

\section{Convergence of the Interior Boundary Condition}
\label{sec:convergence_intbc}
In this section, we assume that the exact adjoint solution can be given as the small-viscosity limit of a viscous adjoint solution. This is a reasonable assumption as already indicated in \cite{BJ98}. Using a viscosity parameter $\varepsilon > 0$, the viscous primal equation can be written as
\begin{align}
\label{visc1p}
 f(w^{\varepsilon})_x + S(w^{\varepsilon}) &= \varepsilon w^{\varepsilon}_{xx}
\end{align}
including again boundary conditions which are not relevant to this investigation. Standard theory \cite{Oleinik63} shows that in the scalar case, given that $\varepsilon \rightarrow 0$, one has $\we \rightarrow w$ in $L^1$. 
The corresponding dual equation is then 
\begin{align}
\label{visc1a}
 -f'(w^{\varepsilon})^T z^{\varepsilon}_x + S'(w^{\varepsilon})^T z^{\varepsilon} &= \varepsilon z^{\varepsilon}_{xx} + p'(w^{\varepsilon}).
\end{align}

Standard assumptions, which can be proven in the scalar one-dimensional case, on the behavior of $w^{\varepsilon}$ are that it is smooth all over the domain, albeit having in a transition region $[\alpha^-, \alpha^+] := [\alpha - \overline \alpha, \alpha + \overline \alpha]$ a gradient that scales as $\frac{1}{\varepsilon}$. 
Outside this region, we state that the gradient is of order unity, i.e., its order of magnitude is independent of $\varepsilon$. Here $\overline \alpha$ is a parameter that goes, in dependency of $\varepsilon$, to zero.
We furthermore assume, in the spirit of Tadmor \cite{Tadmor}, that the adjoint solution is Lipschitz-continuous at $x = \alpha$. This is in good agreement with the results found by Giles and Pierce \cite{GP97} for the quasi one-dimensional Euler equations. 
With respect to the interior boundary conditions, it is thus interesting what happens with the expression
\begin{align}
 \label{viscousibc}
 \lim_{\varepsilon \rightarrow 0^+} \left( [p(w^{\varepsilon})] - [(z^{\varepsilon})^T S(w^{\varepsilon})] \right),
\end{align}
which, in the limit, should be equivalent to (\ref{dual_boundary_condition}) and thus yield zero.
Of course, only involving smooth functions, (\ref{viscousibc}) does not make sense unless we define what we mean by a jump. A reasonable definition is
\begin{align}
 \label{def_smooth_jump}
 [w^{\varepsilon}] := \int_{\alpha^-}^{\alpha^+}\left( \frac{d}{dx} w^{\varepsilon} \right) \ dx,
\end{align}
which, if $w^{\varepsilon}$ converges towards a function $w$ that is discontinuous at $x = \alpha$, converges towards the jump of $w$.

Let us state the following theorem:
\begin{satz}\label{satz:intbcconvergiert}
 Given that both $w^{\varepsilon}$ and $z^{\varepsilon}$, solutions to (\ref{visc1p}) and (\ref{visc1a}), respectively, are smooth, and that outside a transition region $[\alpha^-, \alpha^+]$, both $z_{x}^{\varepsilon}$ and $w_x^{\varepsilon}$ have orders of magnitude independent of $\varepsilon$, it holds that 
 \begin{align}
  \lim_{\varepsilon \rightarrow 0^+} \left( [p(w^{\varepsilon}] - [(z^{\varepsilon})^TS(w^{\varepsilon})] \right) = 0.
 \end{align}
\end{satz}

\begin{proof} The proof exploits both the equations defining $w^{\varepsilon}$ and $z^{\varepsilon}$, and can in principle in a straightforward manner be written as
\begin{align*}
 [p(w^{\varepsilon})] - [(z^{\varepsilon})^T S(w^{\varepsilon})] 
	&\stackrel{(\ref{def_smooth_jump})} = \int_{\alpha^-}^{\alpha^+} \frac{d}{\dx} \left( p(w^{\varepsilon}) - (z^{\varepsilon})^T S(w^{\varepsilon}) \right) \dx \\
					&\stackrel{\hphantom{(8.29)}}= \int_{\alpha^-}^{\alpha^+} \left( p'(w^{\varepsilon}) - S'(w^{\varepsilon})^T z^{\varepsilon} \right) w^{\varepsilon}_x - (z^{\varepsilon}_x)^T S(w^{\varepsilon}) \dx \\
					&\stackrel{(\ref{visc1a})}= \int_{\alpha^-}^{\alpha^+} \left( -f'(w^{\varepsilon})^T z^{\varepsilon}_x - \varepsilon z^{\varepsilon}_{xx} \right) w^{\varepsilon}_x - (z^{\varepsilon}_x)^T S(w^{\varepsilon}) \dx \\
					&\stackrel{\hphantom{(8.29)}}= \int_{\alpha^-}^{\alpha^+} (z^{\varepsilon}_x)^T \left( -f(w^{\varepsilon})_x + \varepsilon w^{\varepsilon}_{xx} - S(w^{\varepsilon}) \right) \dx - [\varepsilon (z^{\varepsilon}_x)^T w^{\varepsilon}_x]_{\alpha^-}^{\alpha^+} \\
					&\stackrel{(\ref{visc1p})}= -\varepsilon [(z^{\varepsilon}_x)^T w^{\varepsilon}_x]_{\alpha^-}^{\alpha^+} = O(\varepsilon).
\end{align*}
Because we are outside the transition region, the term $[(z^{\varepsilon}_x)^T w^{\varepsilon}_x]$ scales independently of $\varepsilon$. This proves our claim.
\end{proof}

\begin{korollar}
 Under the assumptions that both $w^{\varepsilon}$ and $z^{\varepsilon}$ converge towards $w$ and $z$ pointwise, $z$ fulfills the interior boundary condition in the sense of definition \ref{def:ib}, given that the conditions of Theorem \ref{satz:intbcconvergiert} are met. 
\end{korollar}

\section{Conclusions and Outlook}
We have given a general framework for the derivation of the interior boundary condition in the presence of shocks. 
We have furthermore proven that for those stationary balance laws whose solutions can be defined as small-viscosity limits, the adjoint equation fulfills this interior boundary condition. 
In particular, this might explain why low-order, i.e., very diffusive schemes, have no problem in converging towards the correct adjoint solution. It does, however, not explain why some schemes need diffusion over-proportional to the mesh size (see \cite{GU08, GU08Teil2} in the context of time-dependent equations). Numerical evidence shows \cite{SMN10} that this problem also occurs in the steady-state case we are considering here. It might be worth applying our concepts to numerical schemes to get more insight into the behavior of the numerical adjoint procedure. 

We see no major obstacles in carrying out the analysis of section \ref{sec:adjoint_shock}  for both multi-dimensional and time-dependent applications. Conceptually, it should be straightforward, although technically more involved as one has to account for merging and forming shocks and multi-dimensional effects. The consistent augmentation of the functional still relies on Rankine-Hugoniot's condition, in the time-dependent as well as in the steady-state case. Also in those settings, one derives interior boundary conditions similar to the one given in definition \ref{def:ib} along the shock-curve. One non-trivial point, however, is to prove Theorem \ref{satz:intbcconvergiert} again in this setting.

We are aware that definition \ref{defsufficientlysmooth} only holds in very special cases and does not, e.g., apply to viscous approximations of conservation laws. Ongoing work is concerned with the extension of this framework to smooth functions with a steep gradient. This, however, needs a different analysis and is beyond the scope of this paper.

\bibliographystyle{abbrv}
\bibliography{087033}

\end{document}